\documentclass[]{article}

\usepackage{graphicx} 
\usepackage[a4paper,left=35mm,right=35mm,top=30mm,bottom=30mm,marginpar=25mm]{geometry}
\usepackage{amsmath}
\usepackage{amsthm}
\usepackage{amssymb}
\usepackage{authblk}
\usepackage{tikz, pgf}
\usepackage{xcolor} 
\usepackage[displaymath,mathlines]{lineno}
\usetikzlibrary{decorations.markings}
\usepackage[english]{babel}
\usepackage[T1]{fontenc}
\usepackage[utf8]{inputenc}
\usepackage{amsfonts}
\usepackage{mathrsfs}
\usepackage{xfrac}
\usepackage{lmodern}
\DeclareFontFamily{OMX}{lmex}{}
\DeclareFontShape{OMX}{lmex}{m}{n}{<-> lmex10}{}
\usepackage{fix-cm}
\usepackage{listings}
\usepackage{fancyvrb}
\usepackage{enumerate}   
\usetikzlibrary{arrows}
\usepackage{float}
\usepackage{subcaption}
\usepackage{hyperref}
\usepackage{bigints}

\renewcommand{\div}{\operatorname{div}}

\newcommand{\R}{{\mathbb{R}}}

\newcommand{\lra}{\longrightarrow}

\newcommand{\ep}{\varepsilon}

\DeclareMathOperator*{\esssup}{ess\,sup}

\newcommand{\cB}{{\mathcal B}}

\newcommand{\cA}{{\mathcal A}}

\newcommand{\cL}{{\mathcal L}}

\newcommand{\cM}{{\mathcal M}}

\renewcommand{\phi}{\varphi}
\renewcommand{\a}{{\alpha}}
\renewcommand{\b}{{\beta}}

\renewcommand{\d}{{\delta}}

\newcommand{\s}{{\sigma}}

\newcommand{\supp}{\mathrm{supp}\;}

\usepackage{fancyhdr}

\newtheorem{theorem}{Theorem}
\numberwithin{theorem}{section}
\newtheorem{proposition}{Proposition}

\newtheorem{lemma}{Lemma}

\newtheorem{definition}{Definition}

\newtheorem{remark}{Remark}
\providecommand{\keywords}[1]{\textbf{Keywords.} #1}
\providecommand{\mathsubjclass}[2]{\textbf{1991 Mathematics Subject Classification.} #1}

\author{Fabio Bagagiolo\thanks{University of Trento, Department of Mathematics, Via Sommarive 14, 38123 Povo (TN).\\
This author was partially supported by MUR-PRIN2020 Project (No. 2020JLWP23) “Integrated
mathematical approaches to socio-epidemiological dynamics”.}, Rossana Capuani\thanks{University of Arizona, Department of Mathematics, 617 N. Santa Rita Ave., Tucson, USA.\\
This author was partially supported by INdAM-GNAMPA Project 2023 GNAMPA CUP\_E53C22001930001 entitled ``Mean field games methods for mobility and sustainable development'' and by the co-financing of European Union-FSE-REACT-UE, PON ``Ricerca e
Innovazione 2014-2020, Azione IV.6 Contratti di ricerca su tematiche Green''.}, Luciano Marzufero\thanks{``La Sapienza'' University of Rome, Department of Basic and Applied Sciences for Engineering, Via Antonio Scarpa 16, 00161 Rome (RM). \\This work was supported by the Gruppo Nazionale per l'Analisi Matematica e le loro Applicazioni (GNAMPA-INdAM).}}
\date{}

\begin{document}

\title{\Large \bf A single player and a mass of agents:\\ a pursuit evasion-like game}
\maketitle
\begin{abstract}
We study a finite-horizon differential game of pursuit-evasion like, between a single player and a mass of agents. The player and the mass directly control their own evolution, which for the mass is given by a first order PDE of transport equation type. Using also an adapted concept of non-anticipating strategies, we derive an infinite dimensional Isaacs equation, and by dynamic programming techniques we prove that the value function is the unique viscosity solution on a suitable invariant subset of a Hilbert space.
\end{abstract}
\par\smallskip\noindent
\mathsubjclass{Primary: 49N70; Secondary: 49N75; 49L12; 49L25; 49N80; 35Q49.}
\par\bigskip\bigskip\bigskip\noindent
\section{Introduction}
%
The goal of this work is the study of a zero-sum differential game between a single player and a population of agents. In particular, we are interested to the corresponding Hamilton-Jacobi-Isaacs equation.
\par
Both the single player and the mass of the agents are moving in $\mathbb{R}^d$, by the following controlled dynamics (respectively: ordinary differential equation for the single player and continuity partial differential equation for the mass):
\begin{equation}
\label{eq:intro_dynamical_systems}
y'(t)=f(y(t),\alpha(t)),\qquad m_t(x,t)+\mbox{div}(\beta(x,t)m(x,t))=0,
\end{equation}
where the dynamics $f:\mathbb{R}^d\times A\lra\mathbb{R}^d$, and the controls $\alpha:[0,T]\lra A$, $\beta:\mathbb{R}^d\times[0,T]\lra\mathbb{R}^d$ are suitable functions. Here, $m:\mathbb{R}^d\times[0,T]\lra[0, +\infty[$ stands for the time-dependent density of the distribution of the mass of the agents. The payoff function is described as
\begin{equation}
\label{eq:intro_cost}
J(x,m_0,t,\alpha,\beta)=\int_t^T\ell\big(y(s),m(\cdot,s),s,\alpha(s),\beta(\cdot,s)\big)ds+\psi\big(y(T),m(\cdot,T)\big),
\end{equation}
where the running cost $\ell$ and the final cost $\psi$ are suitable functions whose dependence on $m$ and $\beta$ is given by their actual interpretations as state-dependent functions $m(\cdot,s):\mathbb{R}^d\lra[0, +\infty[$ and $\beta(\cdot,s):\mathbb{R}^d\lra\mathbb{R}^d$. 
\par
The goal of the single player is to minimize $J$ and the goal of the mass of the agents is to maximize $J$.
\par
Our main motivation for this kind of study is the possible future application to a differential game between two distinct masses of agents, setting the problem in the framework of the mean-field games theory. In such a differential game, the single agent of one of the two masses should play a zero-sum differential game against the other mass, say $m_Y$, and a mean-field game looking at the evolution of the mass it belongs to, say $m_X$. Hence, it should infer the evolution of $m_X$ and, giving it as a datum, play a zero-sum differential game against the other mass $m_Y$, with payoff
$$
\int_t^T\ell\big(y(s),m_X(\cdot,s),m_Y(\cdot,s),\alpha(s),\beta(\cdot,s)\big)ds+\psi\big(y(T),m_X(\cdot,T),m_Y(\cdot,T)\big).
$$
Being $m_X$ as given, the payoff of the game played by the single agent against the mass $m_Y$ is of the type \eqref{eq:intro_cost}. However, a possible formulation and study of such a mean-field (differential) game problem go beyond the purposes of the present article and seem to present other several crucial issues, both from a modeling and analytical point of view.
\par
Then it is natural to start such an investigation project by the study of the game between a single player and a mass. Actually, such a problem is not well investigated in the literature, especially as concerns the derivation and the study of the corresponding Isaacs equation which naturally turns out to be an infinite dimension equation, due to the presence of the state-space dependent function $m(\cdot,s)$. Hence, its mathematical interest goes beyond the possible application to the mean-field games framework, and moreover it has interesting applicative motivations. 
\par
In view of the Dynamic Programming Principle, as it is standard for differential games, we are going to use the concept of non-anticipating strategies (\cite{ellkal, BCD}). In particular, we are concerned with the ``$\inf\sup$'' problem only, that is we study the lower value function
$$
\underbar V(x,m_0,t)=\inf_{\gamma}\sup_{\beta}J(x,m_0,t,\gamma[\beta],\beta).
$$
Here, $\gamma$ is any non-anticipating strategy for the single player, that is a function from the set of the controls $\beta$ for the mass to the set of the controls $\alpha$ for the single player, such that
$$
\beta_1(\cdot,s)=\beta_2(\cdot,s)\ \mbox{a.e. } s\in[t',t'']\ \mbox{as functions from } \mathbb{R}^d\ \mbox{to } \mathbb{R}^d
$$
$$
\Downarrow
$$
$$
\gamma[\beta_1](s)=\gamma[\beta_2](s)\ \mbox{a.e. } s\in[t',t''].
$$
Denoting by $\tilde{\cal B}$ a suitable set of functions from $\mathbb{R}^d$ to itself (see \eqref{campispaziali}), and assuming that, for all controls $\beta$, it is $\beta(\cdot,t)\in\tilde{\cal B}$ for almost every $t$, then, under suitable hypotheses, the Isaacs equation turns out to be
\begin{equation}
\label{eq:intro_Isaacs}
-u_t+\min_{b\in\tilde{\cal B}}\max_{a\in A}\left\{-f(x,a)\cdot D_xu+\langle D_mu, \div(bm)\rangle_{L^2}-\ell(x,m,t,a,b)\right\}=0,
\end{equation}
where the unknown function $u$ is a function of $(x,m,t)\in\mathbb{R}^d\times L^2(\mathbb{R}^d)\times[0,T]$, and $\langle\cdot,\cdot\rangle_{L^2}$ is the scalar product in $L^2(\R^d)$.
\par
In this paper, we prove that $\underbar V$ is the unique viscosity solution $u$ of \eqref{eq:intro_Isaacs} satisfying $u(x,m,T)=\psi(x,m)$.
\par\smallskip
We notice that with similar techniques we can also study the problem for
$$
\overline{\text{V}}(x,m_0,t)=\sup_{\d}\inf_{\a}J(x,m_0,t,\a,\d[\a]),
$$
where $\a$ is the control for the single player and $\d$ are the non-anticipating strategies for the mass. However, from a modeling point of view, it seems less reasonable because it probably gives too much emphasis in the possible centralized feature of the decisions of the mass. Anyway, whenever the Hamiltonian in \eqref{eq:intro_Isaacs} satisfies the standard Isaacs condition, then, by the uniqueness result, we would have the existence of a value of the game.
\par
As far as the domain of the equation \eqref{eq:intro_Isaacs} is concerned, with respect to the triple $(x,m,t)$, we are able to restrict it to a  suitable set $\tilde X$ which is viable under the evolution \eqref{eq:intro_dynamical_systems} and, more important, is compactly embedded in $\mathbb{R}^d\times L^2(\mathbb{R}^d)\times[0,T]$. To this end, we first need a careful analysis of the solution of the continuity equation in \eqref{eq:intro_dynamical_systems}. Indeed, it is usual in the literature to interpret it in the sense of measures: a solution is a time dependent measure $\mu$ on $\mathbb{R}^d$, whose evolution is continuous in a suitable metric space of measures. Here instead, we want to interpret the solution as a function, i.e. the density of the measure $\mu$ when it is absolutely continuous with respect to the Lebesgue measure. It is known, indeed, that whenever the initial datum $\bar\mu$ is a measure with density $\bar m$, then the solution $\mu(t)$ is also a measure with density $m(\cdot,t)$. An important part of the present paper is then a preliminary direct and deep analysis of the evolution $t\longmapsto m(\cdot,t)$, together with the right choice of the admissible set of controls $\beta$ in view of the subsequent study of the infinite dimensional Isaacs equation \eqref{eq:intro_Isaacs}. Hence Section \ref{continuityeq} contains all such necessary results that, even if somehow and sometimes not surprising, we account interesting by themselves and are not well found in the literature.
\par
The originality and the interest of the Isaacs equation \eqref{eq:intro_Isaacs} is given by its infinite dimensional feature which, to the best of the authors' knowledge, was not studied yet in the framework of the viscosity solution theory (we refer the reader to \cite{BCD} for a comprehensive account of viscosity solutions theory in finite dimension and to the seminal papers \cite{crandall1, crandall2} for the beginning of the theory in infinite dimension). We also notice that first order Hamilton-Jacobi equations in infinite dimension, arising from optimal control problems with controlled dynamics given by a first order PDE as the one in \eqref{eq:intro_dynamical_systems}, are not so widely studied in the literature. Finally, we have to strongly point out that the finite horizon/time dependence feature of the differential game being accounted, even in the simpler finite dimensional case, is rather a newness as for dynamic programming and Isaacs equation.
\par\smallskip
The outline of the paper is the following: in Section \ref{continuityeq}, we study the properties of the evolution $t\longmapsto m(\cdot,t)$. In section \ref{dynprogreg}, we introduce the differential game, write down the Dynamic Programming Principle and study the regularity properties of the lower value function. In Section \ref{HJIforV}, we derive the Isaacs equation, we prove that the lower value function is a viscosity solution of the corresponding final value problem and we finally prove the comparison theorem which gives the uniqueness of the solution. In Section \ref{esempio1dim} we add some comments on two one-dimensional examples.
\par\smallskip
The present Introduction then briefly concludes with a related literature review and comparisons with it.
\par
Recently there has been some interest on Hamilton-Jacobi equations arising from optimal control problems and differential games of pursuit-evasion type, which concern masses and/or measures evolutions and the corresponding controlled continuity equation, see for instance \cite{cardquin, gangbo, marquin, cossopham, jimmarquin, moonbasar, cavmarquin, marigondacapuani}. In particular in \cite{marquin} a Mayer-type differential game between two masses, somehow similar to the present one (even if with just one player and one mass), is studied. In that paper, as in \cite{cardquin, gangbo, cossopham, jimmarquin, moonbasar, cavmarquin, marigondacapuani}, the authors work in the space of probability measures with Wasserstein distance (see \cite{AmbrGigli} for a comprehensive account). This fact poses the use of some very weakened concepts for the solution of the continuity equation like the evolutions of measures concentrated on the trajectories. Despite the generality of such concepts, our model, although less weakened, goes to the direction of a Hilbert functional analysis, which, besides its complete novelty, may be anyway useful especially for possible applications and the synthesis of the optimal feedback. We notice that our stronger regularity of the solutions just comes from the regularity assumptions (in space but not in time) on the available controls for the mass (see \eqref{vectorfields}). In particular, the regularity in time of the controls is the standard and necessary one for the validity of the Dynamic Programming Principle and for the derivation of the corresponding Hamilton-Jacobi equation (for example the possibility of gluing, at a fixed instant, different controls). On the contrary, the assumed regularity in space is a choice of the model, as a possible free choice of the agents. It does not affect the derivation and the validity of the equations, it may be closer to possible real applications and  it allows to study the Hamilton-Jacobi-Isaacs equation in an Hilbert space, which is, as already said and due to its infinite dimensional feature, one of the major novelties of the paper. We remark that the two examples in \S5 suggest the use, for the mass, of controls that are continuous in space. In some particular cases and under suitable hypotheses, the more regular problem can then be seen as an approximation of a less regular problem.

\par
As far as the authors know, this is the first time that a similar approach is adopted. Compared to \cite{marquin}, our model also differs from the presence of a running cost depending on all the state, time and control variables.
\par
Other models of differential games, also in a mean-field type framework, dealing with continuity equations and measures evolutions can be found in \cite{achdou, cirantverzini, hung, aver, colombogaravello, carmona}. We also refer the reader to \cite{sun, ibragimov1, ibragimov2} for further pursuit-evasion like games of multiple pursuers and one evader and to \cite{carmonazhu, lasrylions, cardcirpor} for mean-field games with a major player. Possible applications are investigated in \cite{kolokoltsov1, kolokoltsov2} as regards computer science and cyber-security, and in \cite{ZhouXu, wang} as regards engineering and machine learning.

\section{On the continuity equation for the mass}
\label{continuityeq}
The evolution of a mass on $\R^d$ can be represented by the evolution $t\longmapsto\mu(t)$ of a measure on $\R^d$, where the quantity of mass in $A\subset\R^d$ at time $t$ is the measure of $A$ at the same time: $\mu(t)(A)$. When the mass population is moving according to a given time-dependent vector field $\beta:\R^d\times[0, T]\lra\R^d$, then, at least formally, the evolution of the measure $\mu$ satisfies the following so-called continuity equation
\begin{equation}
\label{conteq}
\begin{cases}
\frac{\partial\mu}{\partial s}(x, s) + \div(\beta(x,s)\mu(x, s))=0,&(x, s)\in\mathbb{R}^d\times]t, T[\\
\mu(\cdot, t)=\bar{\mu}
\end{cases},
\end{equation}
where $T>0$ is the final time, $\bar\mu$ is the value of the measure at the initial time $t$ and, here and in the sequel, $\div$ means the spatial divergence. In \eqref{conteq} we have included the variable $x$ in the notation of $\mu$ with the meaning that for any time $s$, it is a measure on $\R^d$ whose elements are denoted by $x$. Equation \eqref{conteq} comes from conservation of mass arguments and for a general account, in particular for its possible weak formulations in the sense of measures, we refer the reader to \cite{DiPernaLions} and \cite{Amb2004, Amb2008, AmbCri2014}.
\par
The well-posedness of \eqref{conteq} is strongly related to well-posedness of the Cauchy problem
\begin{equation}
\label{ode}
\begin{cases}
y'(s)=\beta(y(s), s),&s\in]t, T[\\
y(t)=x
\end{cases}.
\end{equation}
In the sequel, we will denote the flow associated to \eqref{ode} by $\Phi(x, t, s)$, which is defined as $y(s)$ where $y$ is the solution of \eqref{ode} with the given $(x, t)\in\R^d\times[0, T]$. 
As it is known, if $\beta\in L^1([0, T], W^{1, \infty}(\R^d, \R^d))$, then system \eqref{ode} has a unique solution for any initial data $(x, t)$. For our future purposes, here we assume the stronger following hypothesis: the time-dependent controls $\beta$ are functions from $[0, T]$ to the set $\tilde\cB$, which, for a given fixed $M>0$, is defined as
\begin{multline}
\label{campispaziali}
\tilde\cB:=\left\{b\in W^{2, \infty}(\R^d, \R^d)\cap H^1(\R^d, \R^d):\right.\\
\left.\|b\|_{L^{\infty}(\R^d, \R^d)}\leq M,\ \|b\|_{H^1(\R^d, \R^d)}\leq M,\ \|\div b\|_{W^{1, \infty}(\R^d)}\leq M\right\},
\end{multline}
and moreover they belong to 
\begin{equation}
\label{vectorfields}
\cB:=\{\beta\in L^2([0, T], W^{2, \infty}(\R^d, \R^d)\cap H^1(\R^d, \R^d)):\beta(\cdot, t)\in\tilde\cB\ \text{ for a.e. }t\}.
\end{equation}
\begin{remark}
\label{remark1}
As we have already said in the Introduction, one of our main goals is to study the problem in a Hilbert setting, and hence \eqref{campispaziali} and \eqref{vectorfields} guarantee a good regularity of the trajectories in Hilbert space. In particular we have $\div\beta\in L^2([0, T], W^{1, \infty}(\R^d))$. Moreover, since $W^{2, \infty}(\R^d, \R^d)\subset C^1(\R^d, \R^d)$, for any $(t, s)\in[0, T]\times]t, T]$ the function $\Phi(\cdot, t, s):\R^d\lra\R^d$ is invertible, Lipschitz continuous and $C^1$, together with its inverse, with the Lipschitz constant independent on $t$ and $\beta\in\cB$.
\par
We recall that, for every $t\in[0, T]$ and $s\in[t, T]$, the inverse of the flow $\Phi(\cdot, t, s)$ is defined as $\Phi^{-1}(x, t, s)=y$, where $y$ is such that, starting from $y$ at time $t$ with the forward flow $\Phi$, we arrive in $x$ at time $s$. Still holding our hypotheses, in the sequel we will use the fact that $\Phi^{-1}(\cdot, t, s)$ is Lipschitz continuous w.r.t. $s$ independently on $x, t, \beta$.
\par
Furthermore, if all the fields $\beta$ have the support contained in the same bounded subset of $\R^d$, then the $L^2$-closure $\overline{\tilde\cB}$ of \eqref{campispaziali} is a compact subset of $L^2$.
\end{remark}
\begin{proposition}
Under hypothesis \eqref{vectorfields}, for any Borel measure $\bar\mu$ the solution of the continuity equation \eqref{conteq} in the sense of distribution is given by the push-forward
\begin{equation}
\label{solutioneqcont}
\mu(\cdot, s)=\Phi(\cdot, t, s)\sharp \bar\mu,
\end{equation}
where, for any Borel measurable set $A\subset\R^d$, $\mu(\cdot, s)(A)=\bar\mu(\Phi^{-1}(\cdot, t, s)(A))$.
\end{proposition}
\begin{proof}
See \cite{AmbCri2014}, Proposition 4.
\end{proof}
In the case when $\bar\mu$ is absolutely continuous w.r.t. the Lebesgue measure $\cL^d$, that is $\bar\mu= \bar m\mathcal{L}^d$, where $\bar m:\R^d\lra\R$ is the density, all the measures $\mu(\cdot, s)$ are absolutely continuous w.r.t. $\mathcal{L}^d$ too, and their density $m(\cdot, s)$ can be explicitly computed as (see \cite{AmbCri2014})
\begin{equation}
\label{representationformula}
m(\cdot, s)=\frac{\bar m(\cdot)}{{|\rm{det}} J \Phi(\cdot, t, s)|} \circ \Phi^{-1}(\cdot, t, s),
\end{equation}
where $J\Phi$ is the Jacobian matrix of $\Phi(\cdot, t, s)$. Still referring to \cite{AmbCri2014}, and also using our hypothesis \eqref{vectorfields}, we have the following inequalities
\begin{equation}
\label{stimedet}
0<e^{-MT}\leq{\rm{det}}J\Phi(x, t, s)\leq e^{MT}\quad\text{for all }t\in[0, T],\ s \in [t,T] \ \text{and}\ x\in\R^d.
\end{equation}
Such estimates come from the fact that the function $s\longmapsto\det J\Phi(x, t, s)$ solves
$$
\begin{cases}
\frac{d}{ds}{\rm{det}}J\Phi(x, t, s)=({\rm{div}}\beta)(\Phi(x, t, s), s){\rm{det}}J\Phi(x, t, s),&s\in]t, T]\\
{\rm{det}}J\Phi(x, t, t)=1
\end{cases},
$$
where in the first line $(\div\beta)(\Phi(x, t, s), s)$ stands for the spatial divergence of $\beta$ calculated in the point $(\Phi(x, t, s), s)\in\R^d\times]t, T]$. In particular it is
$$
{\rm{det}}J\Phi(x, t, s)=e^{\left(\int_t^s(\operatorname{div}\beta)(\Phi(x, t, \tau), \tau)d\tau\right)}.
$$
In the following, by $m(x, s; \beta, t, \bar m)$ we denote the value at $(x, s)\in\R^d\times]t, T]$ of the density of the solution \eqref{solutioneqcont} of \eqref{conteq} with $\beta\in\cB$, where the initial datum at time $t$, $\bar\mu$, has density $\bar m$.
\par
From now on we assume
\begin{equation}
\label{ipotesim0}
\bar m\in H^1(\mathbb R^d)\cap W^{1, \infty}(\R^d).
\end{equation}
\subsection{Spatial and time estimates on the density $m$}
\label{spatialtime}
In this subsection we prove that, under hypotheses \eqref{ipotesim0} and \eqref{vectorfields},
\begin{align}
\label{stimespaziali}
&x\longmapsto m(x, s; \beta, t, \bar m)\text{ belongs to }H^1(\mathbb R^d)\cap W^{1, \infty}(\R^d),\\
\label{continuitam}
&m(\cdot, \cdot; \beta, t, \bar m)\in C^0([t, T], H^1(\R^d)\cap W^{1, \infty}(\R^d)),
\end{align}
for every $t\in[0, T]$ and $s\in[t, T]$, where in \eqref{continuitam} the modulus of continuity is independent of $t, \beta$ and only depends on the norms of $\bar m$. Moreover we prove that
\begin{equation}
\label{dipendenzacont1}
\|m(\cdot, s; \beta, t_1, \bar m^1)-m(\cdot, s; \beta, t_2, \bar m^2)\|_{H^1(\mathbb R^d)}\leq\tilde L\left(|t_1-t_2|+\|\bar m^1(\cdot)-\bar m^2(\cdot)\|_{H^1(\mathbb R^d)}\right),
\end{equation}
for every $t_1,t_2\in[0, T]$, $s\in[\max\{t_1,t_2\}, T]$ and for some $\tilde L$ independent of $t_1,t_2$ and $\beta$.
\par\smallskip
For \eqref{stimespaziali}, if we prove that $\det J\Phi(\cdot, t, s)\in W^{1, \infty}(\R^d)$, by the estimates \eqref{stimedet} and by \eqref{ipotesim0}, we then have that 
$$
\frac{\bar m(\cdot)}{\det J\Phi(\cdot, t, s)}\in H^1(\R^d)\cap W^{1, \infty}(\R^d).
$$
Hence, since $\Phi^{-1}(\cdot, t, s)$ is a $C^1$ invertible function, by a known result (see for example \cite{Brezis}), \eqref{representationformula} gives \eqref{stimespaziali}.
\par
So we prove that $\det J\Phi(\cdot, t, s)\in W^{1, \infty}(\R^d)$:
\begin{multline*}
\|\det J\Phi(\cdot, t, s)\|_{W^{1, \infty}(\R^d)}=\|\det J\Phi(\cdot, t, s)\|_{L^{\infty}(\R^d)}+\|\partial_x\det J\Phi(\cdot, t, s)\|_{L^{\infty}(\R^d)}
\\
=\esssup_{x\in\R^d}|\det J\Phi(x, t, s)|+\esssup_{x\in\R^d}|\partial_x\det J\Phi(x, t, s)|
\\
=\esssup_{x\in\R^d}\left|e^{\left(\int_t^s(\operatorname{div}\beta)(\Phi(x, t, \tau), \tau)d\tau\right)}\right|+\esssup_{x\in\R^d}\left|\partial_x\left(e^{\left(\int_t^s(\operatorname{div}\beta)(\Phi(x, t, \tau), \tau)d\tau\right)}\right)\right|
\\
=\esssup_{x\in\R^d}\left|e^{\left(\int_t^s(\operatorname{div}\beta)(\Phi(x, t, \tau), \tau)d\tau\right)}\right|
\\
+\esssup_{x\in\R^d}\left|e^{\left(\int_t^s(\div\beta)(\Phi(x, t, \tau), \tau)d\tau\right)}\int_t^s\partial_x(\div\beta)(\Phi(x, t, \tau), \tau)\partial_x(\Phi(x, t, \tau))d\tau\right|,
\end{multline*}
where, in general, here and sometimes in the sequel, by $\partial_x$ we denote both the Jacobian matrix $d\times d$ and, when the function is scalar, the gradient $1\times d$. By \eqref{stimedet} and Remark \ref{remark1}, we have that $e^{\left(\int_t^s(\div\beta)(\Phi(x. t, \tau), \tau)d\tau\right)}$, $\partial_x(\div\beta)(\Phi(x, t, \tau), \tau)$ and $\partial_x(\Phi(x, t, \tau))$ are bounded, and hence $\|\det J\Phi(\cdot, t, s)\|_{W^{1, \infty}(\R^d)}<+\infty$. Note also that the derivation w.r.t. $x$ of $(\div\beta)(\Phi(x, t, \tau), \tau)$ is possible because $(\div\beta)(\cdot, \tau)\in W^{1, \infty}(\R^d)$ and $\Phi(\cdot, t, \tau)$ is a $C^1$ invertible function, and the (weak) derivative can be computed with the usual differentiation rules (see again \cite{Brezis}).
\par\smallskip
Furthermore, after similar calculations one can show that for the density $m$ we have the following estimate
\begin{equation}
\label{stimamw}
\|m(\cdot, s; \beta, t, \bar m)\|_{W^{1, \infty}(\R^d)}\leq e^{MT}\sup\{1+\tilde M,L_{\Phi^{-1}}\}\|\bar m(\cdot)\|_{W^{1, \infty}(\R^d)},
\end{equation}
where $\tilde M>0$ is a constant enclosing the bounds for the spatial derivative of $\div\beta$, $\Phi$ and $\Phi^{-1}$, independent of $t, \beta$,
and $L_{\Phi^{-1}}$ is the Lipschitz constant of $\Phi^{-1}(\cdot, t, s)$. 
\par\smallskip
Now we prove \eqref{continuitam}. At first we show that, for every $t\in[0, T]$ and $s\geq t$, the function 
\begin{equation}
\label{percontinuita}
\psi:s\longmapsto\det J\Phi(\cdot, t, s)\ \text{ belongs to }\operatorname{Lip}([t, T], W^{1, \infty}(\R^d))
\end{equation}
with Lipschitz constant independent of $t, \beta$. In this way \eqref{representationformula} gives \eqref{continuitam} (see also Remark \ref{remark1}). Then for every $s_1,s_2\in[t, T]$ we have
\begin{multline*}
\|\psi(s_1)-\psi(s_2)\|_{W^{1, \infty}(\R^d)}=\left\|e^{\int_t^{s_1}(\div\beta)(\Phi(\cdot, t, \tau), \tau)d\tau}-e^{\int_t^{s_2}(\div\beta)\Phi(\cdot, t, \tau), \tau)d\tau}\right\|_{W^{1, \infty}(\R^d)}
\\
=\esssup_{x\in\R^d}\left|e^{\int_t^{s_1}(\div\beta)(\Phi(x, t, \tau), \tau)d\tau}-e^{\int_t^{s_2}(\div\beta)\Phi(x, t, \tau), \tau)d\tau}\right|
\\
+\esssup_{x\in\R^d}\left|\partial_x\left(e^{\int_t^{s_1}(\div\beta)(\Phi(x, t, \tau), \tau)d\tau}\right)-\partial_x\left(e^{\int_t^{s_2}(\div\beta)\Phi(x, t, \tau), \tau)d\tau}\right)\right|
\\
\leq L_1\esssup_{x\in\R^d}\left|\int_t^{s_1}(\div\beta)(\Phi(x, t, \tau), \tau)d\tau-\int_t^{s_2}(\div\beta)(\Phi(x, t, \tau), \tau)d\tau\right|
\\
+\esssup_{x\in\R^d}\left|e^{\int_t^{s_1}(\div\beta)(\Phi(x, t, \tau), \tau)d\tau}\int_t^{s_1}\partial_x(\div\beta)(\Phi(x, t, \tau), \tau)\partial_x(\Phi(x, t, \tau))d\tau\right.
\\
\left.-e^{\int_t^{s_2}(\div\beta)(\Phi(x, t, \tau), \tau)d\tau}\int_t^{s_2}\partial_x(\div\beta)(\Phi(x, t, \tau), \tau)\partial_x(\Phi(x, t, \tau))d\tau\right|
\\
\leq L_1\esssup_{x\in\R^d}\left|\int_{s_2}^{s_1}(\div\beta)(\Phi(x, t, \tau), \tau)d\tau\right|
\\
+\esssup_{x\in\R^d}\left|\left(\int_t^{s_1}\partial_x(\div\beta)(\Phi(x, t, \tau), \tau)\partial_x(\Phi(x, t, \tau))d\tau\right)\left(e^{\int_t^{s_1}(\div\beta)(\Phi(x, t, \tau), \tau)d\tau}\right.\right.
\\
\left.\left.-e^{\int_t^{s_2}(\div\beta)(\Phi(x, t, \tau), \tau)d\tau}\right)\right.
\\
\left.+\left(e^{\int_t^{s_2}(\div\beta)(\Phi(x, t, \tau), \tau)d\tau}\right)\left(\int_{s_2}^{s_1}\partial_x(\div\beta)(\Phi(x, t, \tau), \tau)\partial_x(\Phi(x, t, \tau))d\tau\right)\right|
\\
\leq(L_1M+M_1TL_1M+e^{MT}M_1)|s_1-s_2|,
\end{multline*}
where $L_1>0$ is the Lipschitz constant of $e^{(\cdot)}$ and $M_1>0$ is the bound for the term $\partial_x(\div\beta)(\Phi(x, t, \tau), \tau)\partial_x(\Phi(x, t, \tau))$. Hence \eqref{percontinuita} holds with a Lipschitz constant independent of $t, \beta$. 
For proving \eqref{dipendenzacont1}, we observe that, similarly as the computations for \eqref{stimamw}, we have
\begin{equation}
\label{stimamwH1}
\|m(\cdot, s; \beta, t, \bar m)\|^2_{H^1(\R^d)}\leq3e^{2MT}\sup\{1+\tilde M^2, L^2_{\Phi^{-1}}\}\|\bar m(\cdot)\|^2_{H^1(\R^d)}
\end{equation}
and then, by the linearity of the continuity equation for $m$, \eqref{stimamwH1} gives the desired inequality.
\begin{remark}
\label{osscontinuita}
By \eqref{stimespaziali} and \eqref{campispaziali}, $\div(\beta m)\in L^2(\R^d\times[0, T])$. Given the Lipschitz continuity of $s\longmapsto\Phi^{-1}(x, t, s)$ (see Remark \ref{remark1}), similar computations  as above give $\frac{\partial m}{\partial t}\in L^2(\R^d\times[0, T])$. Moreover, since $\mu$, given by \eqref{solutioneqcont}, is the solution of \eqref{conteq} in the sense of measures, we have that its density $m$ is solution, in the sense of distributions, of $\frac{\partial m}{\partial t}+\div(\beta m)=0$. Hence, since $\frac{\partial m}{\partial t}, \div(\beta m)\in L^2(\R^d\times[0, T])$, we have that $m$ is the unique solution in $L^2(\R^d\times[0, T])$ (see \cite{CrippaTesi}) and, in particular, $\frac{\partial m}{\partial t}=-\div(\beta m)$ as functions in $L^2(\R^d\times[0, T])$. From this, it follows that: if $\phi\in C^1(L^2(\R^d))$, then the map
\begin{align*}
\xi&:[t, T]\lra\R\\
\tau&\longmapsto\xi(\tau)=\phi(m(\cdot, \tau; \beta, t, \bar m))
\end{align*}
is differentiable and $\xi'(\tau)=-\langle D_m\phi, \div(\beta m)\rangle_{L^2(\R^d)}$, where $D_m\phi$ denotes the Fr\'echet derivative of $\phi$ w.r.t. $m$ and $\langle\cdot, \cdot\rangle_{L^2(\R^d)}$ the scalar product in $L^2(\R^d)$. 
\end{remark}
\section{The differential game model: first results}
\label{dynprogreg}
In this section, we introduce the model of the differential game between a single player and a mass of agents.
\par
The controlled dynamics of the single player is given by 
\begin{equation}
\label{ode1}
\begin{cases}
y'(s)=f(y(s), \a(s)),&s\in]t, T]\\
y(t)=x\in\mathbb R^d
\end{cases},
\end{equation}
where $T>0$ is given, $t\in[0, T]$, $f:\R^d\times A\lra\R^d$ ($A\subset\R^m$ compact) is continuous, bounded and Lipschitz continuous w.r.t. $x\in\R^d$ uniformly w.r.t. $a\in A$, i.e. there exists $L>0$ such that
$$
\|f(x, a)-f(z, a)\|\leq L\|x-z\|\quad\text{for every}\ x, z\in\R^d\ \text{and}\ a\in A.
$$
The pair $(x, t)$ is the initial datum and the control is
$$
\a\in\cA(t):=\{\a:[t, T]\lra A:\a\text{ is measurable}\}.
$$
Given the control and the initial datum, we denote by $y_{(x, t)}(\cdot; \a)$ the unique solution of \eqref{ode1} in $[t, T]$, and, if no ambiguity arises, we do not display the dependence on the control $\a$ using the notation $y_{(x, t)}(\cdot)$.
\par
Coherently with Remark \ref{osscontinuita}, the controlled equation for the evolution of the mass is
\begin{equation}
\label{contequationformass}
\begin{cases}
m_s(\cdot, s)+\div(\beta(\cdot, s)m(\cdot, s))=0,&s\in]t, T]\\
m(\cdot, t)=\bar m
\end{cases},
\end{equation}
where $t\in[0, T]$ and, using the same notation as in the previous section, $\bar m\in H^1(\R^d)\cap W^{1, \infty}(\R^d)$ and the control $\beta$ belongs to
$$
\cB(t):=\{\beta\in L^2([t, T], W^{2, \infty}(\R^d, \R^d)\cap H^1(\R^d, \R^d)):\beta(\cdot, s)\in\tilde\cB\ \text{ for a.e. }s\}.
$$
We consider a running cost
\begin{align*}
\ell:\R^d\times H^1(\R^d)\times[0, T]\times \R^m\times L^2(\R^d)&\lra[0, +\infty[\\
(x, m, s, a, b)&\longmapsto\ell(x, m, s, a, b),
\end{align*}
(recall that $A\subset\R^m$ and $\tilde\cB\subset L^2(\R^d)$) bounded, strongly continuous and uniformly strongly continuous w.r.t. $(x, m, t)$ uniformly w.r.t. $(a, b)$, that is there exists a modulus of continuity $\omega_{\ell}$ such that, for any fixed $a, b$ it is
$$
|\ell(x_1, m_1, s_1, a, b)-\ell(x_1, m_2, s_2, a, b)|\leq\omega_{\ell}\left(\|x_1-x_2\|+\|m_1-m_2\|_{H^1(\R^d)}+|s_1-s_2|\right)
$$
for all $(x_1, m_1, s_1), (x_2, m_2, s_2)\in\R^d\times H^1(\R^d)\times[0, T]$.
\par
Similarly we consider a final cost
\begin{align*}
\psi:\R^d\times H^1(\R^d)&\lra[0, +\infty[\\
(x, m)&\longmapsto\psi(x, m),
\end{align*}
bounded and uniformly strongly continuous w.r.t. $(x, m)$.
\par
The corresponding cost functional $J$, for $(x, \bar m, t, \a, \beta)\in\R^d\times H^1(\R^d)\times[0, T]\times\cA(t)\times\cB(t)$, is given by
\begin{multline*}
J(x, \bar m, t, \a, \beta)\\
=\int_t^T\ell(y_{(x, t)}(s), m(\cdot, s; \beta, t, \bar m), s, \a(s), \beta(\cdot, s))ds+\psi(y_{(x, t)}(T), m(\cdot, T; \beta, t, \bar m)).
\end{multline*}
The single player wants to minimize the cost $J$ and the mass $m$ wants to maximize it. As argued in the Introduction, here we consider only the $\inf\sup$ best-worst case for the single player. Hence we introduce the non-anticipating strategies for the single player
$$
\Gamma(t)=\{\gamma:\cB(t)\lra\cA(t): \beta\longmapsto\gamma[\beta]\ \text{non-anticipating}\},
$$
where ``non-anticipating'' means, for all $\tau\in[t, T]$,
$$
\beta_1(\cdot, s)=\beta_2(\cdot, s)\ \text{in $\tilde\cB$ a.e. $s$}\in [t, \tau]\ \Rightarrow \gamma[\beta_1](s)=\gamma[\beta_2](s)\ \text{a.e.}\ s\in[t, \tau].
$$
We then consider the lower value function
\begin{equation}
\label{lowervaluefun}
\underbar{V}(x, \bar m, t)=\inf_{\gamma\in\Gamma(t)}\sup_{\beta\in\cB(t)}J(x, \bar m, t, \gamma[\beta], \beta).
\end{equation}
\subsection{Dynamic Programming Principle and regularity of the value function}
\begin{proposition}
\label{dynprog}
Under the previous hypotheses and \eqref{vectorfields}, the lower value function $\underbar V$ \eqref{lowervaluefun} satisfies the Dynamic Programming Principle: for all $(x, \bar m, t)\in\mathbb R^d\times H^1(\mathbb R^d)\times[0, T]$ and for all $\tau\in[t, T]$, 
\begin{multline}
\label{dpp}
\underbar V(x, \bar m, t)=\inf_{\gamma\in\Gamma(t)}\sup_{\beta\in\cB(t)}\Bigg(\int_t^{\tau}\ell(y_{(x, t)}(s), m(\cdot, s; \beta, t, \bar m), s, \gamma[\beta](s), \beta(\cdot, s))ds\\+\underbar V(y_{(x, t)}(\tau), m(\cdot, \tau; \beta, t, \bar m), \tau)\Bigg),
\end{multline}
where $s\longmapsto y_{(x, t)}(s):=y_{(x, t)}(s; \a)$ is the corresponding solution to \eqref{ode1} with control $\a=\gamma[\beta]$. 
\end{proposition}
\begin{proof}
For all $(\xi, u, \tau)\in\mathbb R^d\times H^1(\mathbb R^d)\times[0, T]$ and for all $\varepsilon>0$ we take $\gamma_{(\xi, u, \tau)}\in\Gamma(\tau)$ such that
\begin{equation}
\label{primoeottimo}
\underbar V(\xi, u, \tau)\geq \sup_{\beta\in \cB(\tau)}J(\xi, u, \tau, \gamma_{(\xi, u, \tau)}[\beta], \beta)-\varepsilon.
\end{equation}
Denote by $w(x, \bar m, t)$ the right-hand side of \eqref{dpp}. Then we have to prove that
\begin{itemize}
\item[$(i)$] $\underbar V(x, \bar m, t)\leq w(x, \bar m, t)$;
\item[$(ii)$] $\underbar V(x, \bar m, t)\geq w(x, \bar m, t)$.
\end{itemize}
For $(i)$, take $\bar\gamma\in\Gamma(t)$ such that
\begin{multline*}
w(x, \bar m, t)\geq\sup_{\beta\in\cB(t)}\Bigg(\int_t^{\tau}\ell(y_{(x, t)}(s), m(\cdot, s; \beta, t, \bar m), s, \bar\gamma[\beta](s), \beta(\cdot, s))ds\\+\underbar V(y_{(x, t)}(\tau), m(\cdot, \tau; \beta, t, \bar m), \tau)\Bigg)-\varepsilon.
\end{multline*}
Taken $\beta\in\cB(t)$ and still denoting by $\beta$ its restriction to $[\tau, T]$, we have $\beta\in\cB(\tau)$. We then define $\tilde\gamma\in\Gamma(t)$ as 
$$
\tilde\gamma[\beta](s)=\begin{cases}\bar\gamma[\beta](s),&s\in[t, \tau]\\ \gamma_{(y_{(x, t)}(\tau), m(\cdot, \tau; \beta, t, \bar m), \tau)}[\beta](s),&s\in[\tau, T]\end{cases},
$$
where $y_{(x, t)}(\tau)$ is the position reached at time $\tau$ by the solution of \eqref{ode1} with control $\a=\bar\gamma[\beta]$. The strategy $\tilde\gamma$ is well-defined and non-anticipating, i.e., it belongs to $\Gamma(t)$. We then have
\begin{multline*}
w(x, \bar m, t)\geq\sup_{\beta\in\cB(t)}\Bigg(\int_t^{\tau}\ell(y_{(x, t)}(s), m(\cdot, s; \beta, t, \bar m), s, \bar\gamma[\beta](s), \beta(\cdot, s))ds\\+\underbar V(y_{(x, t)}(\tau), m(\cdot, \tau; \beta, t, \bar m), \tau)\Bigg)-\varepsilon\\
\geq\sup_{\beta\in\cB(t)}\Bigg(\int_t^{\tau}\ell(y_{(x, t)}(s), m(\cdot, s; \beta, t, \bar m), s, \bar\gamma[\beta](s), \beta(\cdot, s))ds\\+J(y_{(x, t)}(\tau), m(\cdot, \tau; \beta, t, \bar m), \tau, \gamma_{(y_{(x, t)}(\tau), m(\cdot, \tau; \beta, t, \bar m), \tau)}[\beta], \beta)\Bigg)-2\varepsilon\\
=\sup_{\beta\in\cB(t)}J(x, \bar m, t, \tilde\gamma[\beta], \beta)-2\varepsilon\geq\underbar V(x, \bar m, t)-2\varepsilon,
\end{multline*}
and we conclude. 
\par
For $(ii)$, take $\varepsilon>0$ and $\beta_1\in\cB(t)$ such that
\begin{multline*}
w(x, \bar m, t)\leq\int_t^{\tau}\ell(y_{(x, t)}(s), m(\cdot, s; \beta_1, t, \bar m), s, \gamma_{(x, \bar m, t)}[\beta_1](s), \beta_1(\cdot, s))ds\\+\underbar V(y_{(x, t)}(s), m(\cdot, \tau; \beta_1, t, \bar m), \tau)+\varepsilon,
\end{multline*}
where $y_{(x, t)}(\cdot):=y_{(x, t)}(\cdot; \a)$ is the solution of \eqref{ode1} with control $\a=\gamma_{(x, \bar m, t)}[\beta_1]$. We define $\tilde\gamma\in\Gamma(\tau)$ as
$$
\tilde\gamma[\beta](s)=\gamma_{(x, \bar m, t)}[\tilde \beta](s),
$$
where $\tilde\beta\in\cB(t)$ is defined as
$$
\tilde \beta(\cdot, s)=\begin{cases}
\beta_1(\cdot, s),&s\in[t, \tau],\\
\beta(\cdot, s),&\text{otherwise}
\end{cases}.
$$
By the non-anticipating feature of $\gamma_{(x, \bar m, t)}$ and by construction, also $\tilde\gamma$ is non-anticipating, i.e., $\tilde\gamma\in\Gamma(\tau)$. Now, take $\beta_2\in\cB(\tau)$ such that
$$
\underbar V(y_{(x, t)}(\tau), m(\cdot, \tau; \beta_1, t, \bar m), \tau)\leq J(y_{(x, t)}(\tau), m(\cdot, \tau; \beta_1, t, \bar m), \tau, \tilde\gamma[\beta_2], \beta_2)+\varepsilon,
$$
and we get (also recall \eqref{primoeottimo})
\begin{multline*}
\underbar V(x, \bar m, t)\geq J(x, \bar m, t, \gamma_{(x, \bar m, t)}[\tilde\beta], \tilde\beta)-\varepsilon\\
=\int_t^{\tau}\ell(y_{(x, t)}(s), m(\cdot, s; \beta_1, t, \bar m), s, \gamma_{(x, \bar m, t)}[\beta_1](s), \beta_1(\cdot, s))ds\\+J(y_{(x, t)}(\tau), m(\cdot, \tau; \beta_1, t, \bar m), \tau, \tilde\gamma[\beta_2], \beta_2)-\varepsilon\\
\geq\int_t^{\tau}\ell(y_{(x, t)}(s), m(\cdot, s; \beta_1, t, \bar m), s, \gamma_{(x, \bar m, t)}[\beta_1](s), \beta_1(\cdot, s))ds\\+\underbar V(y_{(x, t)}(\tau), m(\cdot, \tau; \beta_1, t, \bar m), \tau)-2\varepsilon\geq w(x, \bar m, t)-3\varepsilon,
\end{multline*}
and we conclude.
\end{proof}
We now address the regularity properties of $\underbar V$. Here, and in the following, we denote by $\omega_{\psi}$ the modulus of continuity of $\psi$ (being $\omega_{\ell}$ the one of $\ell$, as above). We also set $C:=\max\{\|f(z, a)\|:(z, a)\in\R^d\times A\}$.
\begin{proposition}
Under the hypotheses of Proposition \ref{dynprog}, the lower value function $\underbar V$ is bounded and uniformly continuous. 
\end{proposition}
\begin{proof}
For the boundedness, for all $(x, \bar m, t)$ we get
$$
\underbar V(x, \bar m, t)\leq G_1T+G_2,
$$
where $G_1>0$ and $G_2>0$ are the bounds for $\ell$ and $\psi$ respectively. 
\par
We then consider $(x_1, m_1, t_1), (x_2, m_2, t_2)\in \mathbb R^d\times H^1(\mathbb R^d)\times[0, T]$, and take $\varepsilon>0$. Let $\gamma_2\in\Gamma(t_2)$ be such that
$$
\underbar V(x_2, m_2, t_2)\geq\sup_{\beta\in\cB(t_2)}J(x_2, m_2, t_2, \gamma_2[\beta], \beta)-\varepsilon.
$$
We define $\gamma_2^1\in\Gamma(t_1)$ as
$$
\gamma_2^1[\beta](s)=\begin{cases}\gamma_2[\tilde \beta](s)&\text{for any }s\ \text{if }t_2\leq t_1\\
\bar a&\text{for }t_1\leq s\leq t_2\ \text{if }t_1\leq t_2\\
\gamma_2[\beta_{|[t_2, T]}](s)&\text{for }s\geq t_2\ \text{if }t_1\leq t_2
\end{cases},
$$
where $\bar a\in\mathbb R^d$ is any a priori fixed vector and, if $t_2\leq t_1$,
$$
\tilde \beta(\cdot, s)=\begin{cases}
\beta(\cdot, t_1)&\text{if }t_2\leq s\leq t_1\\
\beta(\cdot, s)&\text{if }s\geq t_1
\end{cases}.
$$
We then take $\beta_1\in\cB(t_1)$ such that
$$
J(x_1, m_1, t_1, \gamma_2^1[\beta_1], \beta_1)\geq\sup_{\beta\in\cB(t_1)}J(x_1, m_1, t_1, \gamma_2^1[\beta], \beta)-\varepsilon.
$$
Finally, we define $\beta_1^2\in\cB(t_2)$ such that
$$
\beta_1^2(\cdot, s)=\beta_1(\cdot, s)\ \text{for }s\geq t_2+|t_1-t_2|.
$$
Hence we have
\begin{multline*}
\underbar V(x_1, m_1, t_1)-\underbar V(x_2, m_2, t_2)\\
\leq\sup_{\beta\in\cB(t_1)}J(x_1, m_1, t_1, \gamma_2^1[\beta], \beta)-\sup_{\beta\in\cB(t_2)}J(x_2, m_2, t_2, \gamma_2[\beta], \beta)+\varepsilon\\
\leq J(x_1, m_1, t_1, \gamma_2^1[\beta_1], \beta_1)-J(x_2, m_2, t_2, \gamma_2[\beta_1^2], \beta_1^2)+2\varepsilon.
\end{multline*}
Now, if 
\begin{equation}
\label{stimacosti}
|J(x_1, m_1, t_1, \gamma_2^1[\beta_1], \beta_1)-J(x_2, m_2, t_2, \gamma_2[\beta_1^2], \beta_1^2)|
\end{equation}
is infinitesimal as $|t_1-t_2|+\|x_1-x_2\|+\|m_1(\cdot)-m_2(\cdot)\|_{H^1(\mathbb R^d)}$ goes to zero, then we conclude. 
\par
Assume $t_1\leq t_2$. Then \eqref{stimacosti} is equal to
\begin{multline*}
\Bigg|\int_{t_1}^T\ell(y_{(x_1, t_1)}(s), m(\cdot, s; \beta, t_1, m_1), s, \gamma_2^1[\beta_1], \beta_1(\cdot, s))ds+\psi(y_{(x_1, t_1)}(T), m(\cdot, T; \beta, t_1, m_1))\\
-\int_{t_2}^T\ell(y_{(x_2, t_2)}(s), m(\cdot, s; \beta, t_2, m_2), s, \gamma_2[\beta_1^2](s), \beta_1^2(\cdot, s))ds\\
-\psi(y_{(x_1, t_1)}(T), m(\cdot, T; \beta, t_2, m_2))\Bigg|\\
=\Bigg|\int_{t_1}^{t_2}\ell(y_{(x_1, t_1)}(s), m(\cdot, s; \beta, t_1, m_1), s, \bar a, \beta_1(\cdot, s))ds\\
+\int_{t_2}^T\ell(y_{(x_1, t_1)}(s), m(\cdot, s; \beta, t_1, m_1), s, \gamma_2[\beta_{|[t_2, T]}](s), \beta_1^2(\cdot, s))ds\\
-\int_{t_2}^T\ell(y_{(x_2, t_2)}(s), m(\cdot, s; \beta, t_2, m_2), s, \gamma_2[\beta_{|[t_2, T]}](s), \beta_1^2(\cdot, s))ds\\
+\psi(y_{(x_1, t_1)}(T), m(\cdot, T; \beta, t_1, m_1))-\psi(y_{(x_2, t_2)}(T), m(\cdot, T; \beta, t_2, m_2))\Bigg|\\
\leq G_1|t_1-t_2|+|\psi(y_{(x_1, t_1)}(T), m(\cdot, T; \beta, t_1, m_1))-\psi(y_{(x_2, t_2)}(T), m(\cdot, T; \beta, t_2, m_2))|\\
+\int_{t_2}^T\left|\ell(y_{(x_1, t_1)}(s), m(\cdot, s; \beta, t_1, m_1), s, \gamma_2[\beta_{|[t_2, T]}](s), \beta_1^2(\cdot, s))\right.\\
\left.-\ell(y_{(x_2, t_2)}(s), m(\cdot, s; \beta, t_2, m_2), s, \gamma_2[\beta_{|[t_2, T]}](s), \beta_1^2(\cdot, s))\right|ds\\
\leq G_1|t_1-t_2|\\
+\omega_{\psi}\left(e^{LT}(\|x_1-x_2\|+C|t_1-t_2|)+\|m(\cdot, T; \beta, t_1, m_1)-m(\cdot, T; \beta, t_2, m_2)\|_{H^1(\mathbb R^d)}\right)\\
+T\omega_{\ell}\left(e^{LT}(\|x_1-x_2\|+C|t_1-t_2|)+\|m(\cdot, s; \beta, t_1, m_1)-m(\cdot, s; \beta, t_2, m_2)\|_{H^1(\mathbb R^d)}\right)\\
\leq G_1|t_1-t_2|+\omega_{\psi}\left(e^{LT}(\|x_1-x_2\|+C|t_1-t_2|)+\tilde L\|m_1(\cdot)-m_2(\cdot)\|_{H^1(\R^d)}\right)\\
+T\omega_{\ell}\left(e^{LT}(\|x_1-x_2\|+C|t_1-t_2|)+\tilde L\|m_1(\cdot)-m_2(\cdot)\|_{H^1(\R^d)}\right),
\end{multline*}
where the last inequality holds due to \eqref{dipendenzacont1} in \S\ref{spatialtime}. Therefore, as $|t_1-t_2|+\|z_1-z_2\|+\|m_1(\cdot)-m_2(\cdot)\|_{H^1(\mathbb R^d)}$ goes to zero, \eqref{stimacosti} is infinitesimal. 
\par
Now we assume $t_2\leq t_1$. Then
\begin{multline*}
\Bigg|\int_{t_1}^T\ell(y_{(x_1, t_1)}(s), m(\cdot, s; \beta, t_1, m_1), s, \gamma_2^1[\beta_1], \beta_1(\cdot, s))ds+\psi(y_{(x_1, t_1)}(T), m(\cdot, T; \beta, t_1, m_1))\\
-\int_{t_2}^T\ell(y_{(x_2, t_2)}(s), m(\cdot, s; \beta, t_2, m_2), s, \gamma_2[\beta_1^2](s), \beta_1^2(\cdot, s))ds\\
-\psi(y_{(x_2, t_2)}(T), m(\cdot, T; \beta, t_2, m_2))\Bigg|\\
=\Bigg|\int_{t_1}^T\ell(y_{(x_1, t_1)}(s), m(\cdot, s; \beta, t_1, m_1), s, \gamma_2[\beta_1](s), \beta_1(\cdot, s))ds\\
+\psi(y_{(x_1, t_1)}(T), m(\cdot, T; \beta, t_1, m_1))\\
-\int_{t_2}^{t_1}\ell(y_{(x_2, t_2)}(s), m(\cdot, s; \beta, t_2, m_2), s, \gamma_2[\beta_1^2](s), \beta_1^2(\cdot, s))ds\\
-\int_{t_1}^T\ell(y_{(x_2, t_2)}(s), m(\cdot, s; \beta, t_2, m_2), s, \gamma_2[\beta_1](s), \beta_1(\cdot, s))ds-\psi(y_{(x_2, t_2)}(T), m(\cdot, T; \beta, t_2, m_2))\Bigg|\\
\leq G_1|t_1-t_2|+|\psi(y_{(x_1, t_1)}(T), m(\cdot, T; \beta, t_1, m_1))-\psi(y_{(x_2, t_2)}(T), m(\cdot, T; \beta, t_2, m_2))|\\
+\int_{t_1}^T\left|\ell(y_{(x_1, t_1)}(s), m(\cdot, s; \beta, t_1, m_1), s, \gamma_2[\beta_1](s), \beta_1(\cdot, s))\right.\\
\left.-\ell(y_{(x_2, t_2)}(s), m(\cdot, s; \beta, t_2, m_2), s, \gamma_2[\beta_1](s), \beta_1(\cdot, s))\right|ds\\
\leq G_1|t_1-t_2|\\
+\omega_{\psi}\left(e^{LT}(\|x_1-x_2\|+C|t_1-t_2|)+\|m(\cdot, T; \beta, t_1, m_1)-m(\cdot, T; \beta, t_2, m_2)\|_{H^1(\mathbb R^d)}\right)\\
+T\omega_{\ell}\left(e^{LT}(\|x_1-x_2\|+C|t_1-t_2|)+\|m(\cdot, s; \beta, t_1, m_1)-m(\cdot, s; \beta, t_2, m_2)\|_{H^1(\mathbb R^d)}\right)\\
\leq G_1|t_1-t_2|+\omega_{\psi}\left(e^{LT}(\|x_1-x_2\|+C|t_1-t_2|)+\tilde L\|m_1(\cdot)-m_2(\cdot)\|_{H^1(\R^d)}\right)\\
+T\omega_{\ell}\left(e^{LT}(\|x_1-x_2\|+C|t_1-t_2|)+\tilde L\|m_1(\cdot)-m_2(\cdot)\|_{H^1(\R^d)}\right),
\end{multline*}
where in the last inequality we have used again \eqref{dipendenzacont1}. Letting $|t_1-t_2|+\|z_1-z_2\|+\|m_1(\cdot)-m_2(\cdot)\|_{H^1(\mathbb R^d)}$ go to zero, we conclude.
\end{proof}
\section{The Hamilton-Jacobi-Isaacs equation for $\underbar V$}
\label{HJIforV}
The value function $\underbar V$ is a function of $(x, m, t)\in\R^d\times H^1(\R^d)\times[0, T]$. Here we want to derive the corresponding Hamilton-Jacobi-Isaacs equation and consider it in a suitable set $\tilde X$ for the variables $(x, m, t)$, which is compact in $\R^d\times L^2(\R^d)\times[0, T]$ and invariant for the controlled evolutions $y_{(x, t)}(\cdot; \a)$, $m(\cdot, s; \beta, t, \bar m)$. In order to determine such a possible set $\tilde X$, we first fix an open bounded subset $\Omega\subset\R^d$ which is going to contain the supports of all the admissible initial distribution $m_0$, at time $t=0$, for the mass. Moreover, we fix a constant $K>0$ such that all the admissible initial distributions of the mass belong to
$$
{\cal M}=\left\{m_0\in W^{1,\infty}(\Omega):\|m_0\|_{W^{1,\infty}(\Omega)}\leq K,\ \supp m_0\subset\Omega\right\}.
$$
Referring to \eqref{stimamw}, we set $B:=\max\{1+\tilde M,L_{\Phi^{-1}}\}K$ and for all $m_0\in{\cal M}$, $\beta\in\cB(0)$ and $s\in[0,T]$ we get
\begin{equation}
\label{stimanuovanuova}
\|m(\cdot,s;\beta, 0, m_0)\|_{W^{1,\infty}(\Omega)}\le Be^{MT}.
\end{equation}
Now, denoting by $M$ a bound for $\|\beta(\cdot,t)\|_\infty$, as in \eqref{campispaziali}, i.e. the maximal possible velocity of the mass' agents, we define, for all $t\in[0,T]$ and for all $s\in[t,T]$, the bounded sets
$$
\Omega_1(t)=\overline B(\Omega, Mt),\ \ \Omega_1(s,t)=\overline B\left(\Omega_1(t),M(s-t)\right),
$$
being, as example, $\overline{B}(\Omega, Mt)=\{x\in\R^d:\text{dist}(x, \Omega)\leq Mt\}$. Note that $\Omega_1(t)$ represents the maximal set that can be invaded, up to the time $t$, by agents starting from $\overline\Omega$ at the time $t_0=0$, whereas $\Omega_1(s,t)$ represents the maximal set that can be invaded, at the time $s$, by agents starting from $\Omega_1(t)$ at the time $t$. Of course, we have a sort of semigroup property: $\Omega_1(0)=\overline\Omega$, $\Omega_1(s,t)=\Omega_1(s)$ for all $t\in[0,s]$ and, in particular, $\Omega_1(T)=\Omega_1(T,t)$ is the maximal set that can be invaded at the final time $T$ by a mass having support in $\Omega_1(t)$ at the time $t$. That is, assuming $\bar m\in W^{1, \infty}(\R^d)\cap H^1(\R^d)$ with $\supp\bar m\subset\Omega_1(t):=\overline{B}(\Omega, r(t))$, we have, independently of $\beta\in\cB(t)$,
$$
\supp m(\cdot, s'; \beta, \bar m)\subset\Omega_1(s',t)\subset\Omega_1(s'',t)\subset\Omega_1(T,t)=\Omega_1(T)\label{stimanuova0}\\
$$
for all $t\le s'\le s''\le T$. Note that all such functions $m$ can be assumed as defined in the whole set $\Omega_1(T)$ putting them equal to $0$ outside their support. Now, for every $t\in[0, T]$, let $K(t)\subset H^1(\Omega_1(T))$ be a compact set such that $K(t')\subset K(t)$ for every $0\leq t'\leq t\leq T$, $m\in K(t)$ if $K(t_n)\ni m_n\lra m$ in $H^1(\Omega_1(T))$ and $t_n\to t$ in $[0, T]$, and for every $0\leq t'\leq t\leq T$, $K(t)$ contains $m(\cdot, t; \beta, t', \bar m)$ for every $\beta\in\cB$ and $\bar m\in K(t')$. Observe that, by \eqref{dipendenzacont1} and suitable possible convergences on the set of controls $\beta$, a simple example of such compact sets $K(t)$ may be the points of all trajectories up to the time $t$ started at time $0$ from points of a compact set $K$ in $H^1(\Omega_1(T))$. Recalling \eqref{stimanuovanuova}, then for all $t\in[0,T]$, we define the following set:
\begin{multline*}
X(t):=\left\{\bar m\in W^{1, \infty}(\Omega_1(T))\cap K(t):\right.\\
\left.\supp\bar m\subset\Omega_1(t),\ \|m(\cdot, s; \beta, t, \bar m)\|_{W^{1, \infty}(\Omega_1(T))}\leq Be^{MT} \ \text{for all }s\in[t,T]\text{ and }\beta\in\cB(t)\right\}.
\end{multline*}
In this way, the domain of the Hamiltonian we are going to consider (see \eqref{hamiltonian} below) concerning the pair $(m, t)$ is defined as the following set:
\begin{multline*}
X=\bigcup_{t\in[0, T]}(X(t)\times\{t\})
=\left\{(\bar m, t)\in W^{1, \infty}(\Omega_1(T))\cap K(t)\times[0, T]:\right.\\
\left.\supp\bar m\subset\Omega_1(t),\ \|m(\cdot, s; \beta, t, \bar m)\|_{W^{1, \infty}(\Omega_1(T))}\leq Be^{MT}\ \text{for all }s\in[t,T] \text{ and }\beta\in\cB(t)\right\}.
\end{multline*}
By construction, $X$ is compact in $H^1(\Omega_1(T))\times[0, T]$ and hence also in $L^2(\Omega_1(T))\times[0, T]$. In particular, let us directly prove that it is closed in $H^1\times[0, T]$. Let $(m_n,t_n)\in X$ be converging in $H^1\times[0,T]$ to $(\bar m,\bar t)$. We want to show that $(\bar m,\bar t)\in X$. Since $(m_n,t_n)\in X$, the sequence $m_n$ is bounded in $W^{1,\infty}$ and hence the convergence to $m$ is also uniform, from which we obtain $\supp\bar m\subset\Omega_1(\bar t)$. Now, take $s>\bar t$ and $n$ sufficiently large such that $t_n<s$. By \eqref{continuitam} and \eqref{dipendenzacont1} (the modulus of continuity of \eqref{continuitam} is independent of $m_n\in X$), we have the convergence in $H^1$ of $m(\cdot,s; \beta, t_n, m_n)$ to $m(\cdot,s;\beta, \bar t, \bar m)$. The boundedness of the former in $W^{1,\infty}$ implies the weak-star convergence, which implies
$$
\|m(\cdot,s;\beta,\bar t, \bar m)\|_{W^{1,\infty}}\le\liminf\|m(\cdot,s;\beta, t_n, m_n)\|_{W^{1,\infty}}\le Be^{MT}.
$$
\par
Due to compactness, the $H^1$ and $L^2$ norms are topologically equivalent in $X$, that is there exists a modulus of continuity $\tilde\omega$ such that for every $(m_1,t_1),( m_2,t_2)\in X$, 
\begin{equation}
\label{modcontm}
\|m_1-m_2\|_{H^1(\Omega_1(T))}\leq\tilde\omega\left(\|m_1-m_2\|_{L^2(\Omega_1(T))}\right).
\end{equation}
Moreover, $X$ contains all the possible trajectories $s\longmapsto(m(\cdot,s;\beta,t,\bar m),s)=(m(s),s)$ (with field $\beta\in\cB(t)$) starting from any $(\bar m,t)\in X$. Indeed, for all $s\in[t,T]$ and for $\tau\in[s,T]$, it is
$$
\|m(\cdot,\tau;\beta,s,m(s))\|_{W^{1, \infty}(\Omega_1(T))}=\|m(\cdot,\tau;\beta,t,\bar m)\|_{W^{1, \infty}(\Omega_1(T))}\le Be^{MT}
$$
and hence $(m(s),s)\in X$ because, by definition, $\supp m(s)\subset\Omega_1(s)$ since $\supp\bar m\subset\Omega_1(t)$. Hence $X$ is viable under the evolution of the continuity equation with fields $\beta\in\cB(t)$. In particular, it contains all the trajectories $t\longmapsto(m(t),t)$ starting from $(m_0,0)$ with $m_0\in\cM\cap K(0)$.
\par
A similar construction is naturally done for the trajectories of the single player of \eqref{ode1}. More precisely, we define $\Omega_2(t):=\overline{B}(\Omega, Ct)$, where $C$ is the bound for $\|f(\cdot, a)\|$, and hence we have that $y_{(x, t)}(\tau)\in\Omega_2(s,t):=\overline{B}(\Omega_2(t), C(s-t))$ for every $\tau\in[t, s]$. We define then the set
\begin{multline*}
\tilde X:=\bigcup_{t\in[0, T]}(\Omega_2(t)\times X(t)\times\{t\})=\\
\left\{(x, m, t)\in\Omega_2(T)\times W^{1, \infty}(\Omega_1(T))\cap K(t)\times[0, T]:x\in\Omega_2(t),\ (m,t)\in X\right\}.
\end{multline*}
Hence $\tilde X$ is a compact set in $\R^d\times H^1\times[0,T]$, and is viable under all the admissible trajectories with controls $\alpha\in\cA(t)$ and $\beta\in\cB(t)$. That is, all our admissible trajectories in the time interval $[0,T]$, when starting from points of $\tilde X$, can not exit from $\tilde X$ itself. From the point of view of the Hamilton-Jacobi-Isaacs equation that we are going to state and study on $\tilde X$, this means that, besides the time boundary at $t=0$ and $t=T$, there will not be other spatial boundaries requiring boundary conditions. Indeed, from any point of $\tilde X$, each trajectory still remains inside $\tilde X$ itself, then the Dynamic Programming Principle always holds and so definitely the Hamilton-Jacobi-Isaacs equation does, without spatial boundary conditions (see the proof of Theorem \ref{esistenza}).

We then define the Hamiltonian $H:\tilde X\times\R^d\times L^2(\R^d)\lra\R$ as
\begin{equation}
\label{hamiltonian}
H(x, m, t, p, q):=\min_{b\in\tilde\cB}\max_{a\in A}\left\{-f(x, a)\cdot p+\langle q, \operatorname{div}(bm)\rangle_{L^2(\R^d)}-\ell(x, m, t, a, b)\right\},
\end{equation}
where $\langle \cdot, \cdot\rangle_{L^2(\R^d)}$ denotes the scalar product in $L^2(\R^d)$. In the sequel, we denote by $D_x$ the gradient $1\times d$, and we are going to consider the fields $\beta(\cdot, s)$ defined on the compact set $\Omega_1(T)$, i.e., $\beta:\Omega_1(T)\times[t, T]\lra\R^d$.
\begin{definition}
\label{viscositysolutiondef}
A function $u\in C^0(\tilde X)$ is a viscosity subsolution of 
\begin{equation}
\label{hjbviscosity}
-u_t(x, m, t)+H(x, m, t, D_xu, D_mu)=0\quad\text{in }\tilde X
\end{equation}
if, for any $\phi\in C^1(\R^d\times L^2(\R^d)\times[0, T])$,
$$
-\phi_t(\bar x, \bar m, \bar t)+H(\bar x, \bar m, \bar t, D_x\phi(\bar x, \bar m, \bar t), D_m\phi(\bar x, \bar m, \bar t))\leq0
$$
at any local maximum point $(\bar x, \bar m, \bar t)\in\tilde X$ of $u-\phi$. Similarly, $u\in C^0(\tilde X)$ is a viscosity supersolution of \eqref{hjbviscosity} if, for any $\phi\in C^1(\R^d\times L^2(\R^d)\times[0, T])$,
$$
-\phi_t(\tilde x, \tilde m, \tilde t)+H(\tilde x, \tilde m, \tilde t, D_x\phi(\tilde x, \tilde m, \tilde t), D_m\phi(\tilde x, \tilde m, \tilde t))\geq0
$$
at any local minimum point $(\tilde x, \tilde m, \tilde t)\in\tilde X$ of $u-\phi$. Finally, $u$ is a viscosity solution of \eqref{hjbviscosity} if it is simultaneously a viscosity sub- and supersolution. 
\end{definition}
We point out that the above local maximum/minimum point is w.r.t. $\tilde X$ and not necessarily w.r.t. $\R^d\times L^2(\R^d)\times[0, T]$, where the test function is defined. This is not an issue because, as argued before, $\tilde X$ is viable for our trajectories and moreover in $\tilde X$ the convergences of $m$ in $H^1(\R^d)$ and in $L^2(\R^d)$ are equivalent.
\begin{theorem}
\label{esistenza}
Under the hypotheses of \S\ref{dynprogreg}, the lower value function $\underbar V$ is a viscosity solution (see Definition \ref{viscositysolutiondef}) of
\begin{equation}
\label{HJI}
\begin{cases}
-V_t+H(x, m, t, D_xV, D_mV)=0&\text{in }\tilde X\\
V(x, m, T)=\psi(x, m),&(x, m)\in\Omega_2(T)\times X(T)\times\{T\}
\end{cases}.
\end{equation}
\end{theorem}
For the proof, we need the following
\begin{lemma}
\label{lemmaxstrategia}
Assume the hypotheses of Theorem \ref{esistenza}. Let $(x, m, t)\in\tilde X$ and $\phi\in C^1(\R^d\times L^2(\R^d)\times[0, T])$ be such that
\begin{equation}
\label{primaipotesilemma}
-\phi_t(x, m, t)+H(x, m, t, D_x\phi(x, m, t), D_m\phi(x, m, t))=\theta>0.
\end{equation}
Then there exists $\gamma^*\in\Gamma(t)$ such that for all $\beta\in\cB(t)$ and $\tau>t$ sufficiently close to $t$,
\begin{multline}
\label{tesilemma}
\int_t^{\tau}\Big\{\ell(y_{(x, t)}, m(\cdot, s; \beta, t, \bar m), s, \gamma^*[\beta], \beta(\cdot, s))+f(y_{(x, t)}, \gamma^*[\beta])\cdot D_x\phi(y_{(x, t)}, m(\cdot, s; \beta, t, \bar m), s)\\-\langle D_m\phi(y_{(x, t)}, m(\cdot, s; \beta, t, \bar m), s), \operatorname{div}(\beta(\cdot, s)m(\cdot, s; \beta, t, \bar m))\rangle_{L^2(\Omega_1(T))}\\
+\phi_t(y_{(x, t)}, m(\cdot, s; \beta, t, \bar m), s)\Big\}ds\leq-\frac{\theta(\tau-t)}{4},
\end{multline}
where $y_{(x, t)}=y_{(x, t)}(\cdot, \gamma^*[\beta])$. 
\end{lemma}
\begin{proof}
We define
\begin{multline*}
\Lambda(z, m, t, a, b):=-\phi_t(z, m, t)-f(z, a)\cdot D_x\phi(z, m, t)\\+\langle D_m\phi(z, m, t), \operatorname{div}(bm)\rangle_{L^2(\Omega_1(T))}-\ell(z, m, t, a, b).
\end{multline*}
By definition of $H$ and \eqref{primaipotesilemma}, we have
$$
\min_{b\in\tilde\cB}\max_{a\in A} \Lambda(x, m, t, a, b)=\theta.
$$
Then, for any $b\in\tilde\cB$ there is $a=a(b)\in A$ such that
$$
\Lambda(x, m, t, a, b)\geq\theta. 
$$
Observe that $\Lambda(x, m, t, a, \cdot)$ is uniformly continuous in $\overline{\tilde\cB}$ w.r.t. the $L^2$ topology, where we recall that $\overline{\tilde\cB}$ is the closure of $\tilde\cB$ (see Remark \ref{remark1}). This gives that
$$
\Lambda(z, m, t, a, \zeta)\geq\frac{3\theta}{4}\ \ \text{for all}\ \zeta\in B(b, r)\cap\overline{\tilde\cB}
$$
for some $r=r(b)>0$. Since $\overline{\tilde\cB}$ is compact in $L^2(\Omega_1(T), \mathbb R^d)$ (see Remark \ref{remark1} again), there exist finitely many points $b_1,\ldots,b_n\in\overline{\tilde\cB}$ and $r_1,\ldots, r_n>0$ such that
\begin{equation}
\label{factlemma1}
\overline{\tilde\cB}\subset\bigcup_{i=1}^nB(b_i, r_i)\cap\overline{\tilde\cB},
\end{equation}
and, for $a_i=a(b_i)$,
$$
\Lambda(x, m, t, a_i, \zeta)\geq\frac{3\theta}{4}\ \ \text{for all}\ \zeta\in B(b_i, r_i)\cap\overline{\tilde\cB}.
$$
Since $\tilde\cB$ is dense in $\overline{\tilde\cB}$, we can assume, without loss of generality, that \eqref{factlemma1} holds, for suitable $r_1,\ldots,r_n>0$, with the centers $b_i\in\tilde\cB$ for every $i=1,\ldots,n$. 
Now define $\Psi:\tilde\cB\lra A$ by setting
$$
\Psi(b)=a_k\ \ \text{if}\ b\in B(b'_k, r'_k)\setminus\bigcup_{i=1}^{k-1}B(b'_i, r'_i).
$$
Observe that for any $\beta\in\cB(t)$ and $s\in[t, T]$, we have that $\Psi(\beta(\cdot, s))$ is measurable as the composition of two measurable functions, and then
we can define $\gamma^*\in\Gamma(t)$ by setting 
$$
\gamma^*[\beta](s):=\Psi(\beta(\cdot, s)).
$$
By definition of $\Psi$, we have
$$
\Lambda(x, m, t, \Psi(b), b)\geq\frac{3\theta}{4}\quad\text{for all}\ b\in\tilde\cB,
$$
and by the continuity of $\Lambda(x, m, t, a, \cdot)$, by general estimates on the trajectory $y_{(x, t)}$ and by \eqref{continuitam}, there exists $\tau>t$ such that
\begin{equation}
\label{ultimalambda}
\Lambda(y_{(x, t)}(s), m(\cdot, s; \beta, t, \bar m), s, \gamma^*[\beta](s), \beta(\cdot, s))\geq\frac{\theta}{2}\quad\text{for}\ t\leq s< \tau \ \text{and for each } \beta\in\cB(t).
\end{equation}
By integrating both sides of \eqref{ultimalambda} from $t$ to $\tau$, we obtain \eqref{tesilemma} for $\tau$ sufficiently close to $t$.
\end{proof}
\begin{proof}[Proof of Theorem \ref{esistenza}]
At first we prove that $\underbar V$ is a subsolution (see Definition \ref{viscositysolutiondef}) of \eqref{HJI}. Let $\phi\in C^1(\R^d\times L^2(\R^d)\times[0, T])$ and $(\bar x, \bar m, \bar t)$ be a local maximum point w.r.t. $\tilde X$ for $\underbar V-\phi$, and $\underbar V(\bar x, \bar m, \bar t)=\phi(\bar x, \bar m, \bar t)$. We stress again that any trajectory starting from $(\bar x, \bar m, \bar t)$ does not exit from $\tilde X$. We assume by contradiction that \eqref{primaipotesilemma} holds and then, by Lemma \ref{lemmaxstrategia}, there exists $\gamma^*\in\Gamma(\bar t)$ such that, for all $\beta\in\cB(\bar t)$ and all $\tau$ sufficiently close to $\bar t$,
\begin{multline}
\label{perparti}
\int_{\bar t}^{\tau}\ell(y_{(\bar x, \bar t)}, m(\cdot, s; \beta, \bar t, \bar m), s, \gamma^*[\beta], \beta(\cdot, s))ds\\
+\phi(y_{(\bar x, \bar t)}(\tau), m(\cdot, \tau; \beta, \bar t, \bar m), \tau)-\phi(\bar x, \bar m, \bar t)\leq-\frac{\theta(\tau-\bar t)}{4},
\end{multline}
where $y_{(\bar x, \bar t)}=y_{(\bar x, \bar t)}(\cdot, \gamma^*[\beta])$. Since $\underbar V-\phi$ has a local maximum at $(\bar x, \bar m, \bar t)$ and $\underbar V(\bar x, \bar m, \bar t)=\phi(\bar x, \bar m, \bar t)$, by classical estimates on the trajectory $y_{(\bar x, \bar t)}$ and by \eqref{continuitam}, we have
$$
\phi(y_{(\bar x, \bar t)}(\tau), m(\cdot, \tau; \beta, \bar t, \bar m), \tau)-\phi(\bar x, \bar m, \bar t)\geq \underbar V(y_{(\bar x, \bar t)}(\tau), m(\cdot, \tau; \beta, \bar t, \bar m), \tau)-\phi(\bar x, \bar m, \bar t)
$$
for $\tau$ sufficiently close to $\bar t$. Plugging this into \eqref{perparti}, we obtain
\begin{multline*}
\inf_{\gamma\in\Gamma(\bar t)}\sup_{\beta\in\cB(\bar t)}\left\{\int_{\bar t}^{\tau}\ell(y_{(\bar x, \bar t)}, m(\cdot, s; \beta, \bar t, \bar m), s, \gamma^*[\beta], \beta)ds+\underbar V(y_{(\bar x, \bar t)}(\tau), m(\cdot, \tau; \beta, \bar t, \bar m), \tau)\right\}
\\
-\underbar V(\bar x, \bar m, \bar t)\leq-\frac{\theta(\tau-\bar t)}{4}<0,
\end{multline*}
which contradicts the inequality ``$\leq$'' in \eqref{dpp}. Then $\underbar V$ is a subsolution (see Definition \ref{viscositysolutiondef}) of \eqref{HJI}.
\par\smallskip
Next we show that $\underbar V$ is a supersolution (see Definition \ref{viscositysolutiondef}) of \eqref{HJI}. Let $\phi\in C^1(\R^d\times L^2(\R^d)\times[0, T])$ and $(\bar x, \bar m, \bar t)$ be a local minimum point for $\underbar V-\phi$, and $\underbar V(\bar x, \bar m, \bar t)=\phi(\bar x, \bar m, \bar t)$. Suppose by contradiction that
$$
-\phi_t(\bar x, \bar m, \bar t)+H(\bar x, \bar m, \bar t, D_x\underbar V, D_m\underbar V)=-\theta<0.
$$
By the definition of $H$, there exists $b^*\in\tilde\cB$ such that
$$
-\phi_t(\bar x, \bar m, \bar t)-f(\bar x, a)\cdot D_x\phi(\bar x, \bar m, \bar t)+\langle D_m\phi(\bar x, \bar m, \bar t), \div(b^*\bar m)\rangle_{L^2(\Omega_1(T))}-\ell(\bar x, \bar m, \bar t, a, b^*)\leq -\theta
$$
for all $a\in A$. For $\tau$ sufficiently close to $\bar t$ and any $\gamma\in\Gamma(\bar t)$, we have
\begin{multline*}
-\phi_t(y_{(\bar x, \bar t)}(s), m(\cdot, s; b^*, \bar t, \bar m), s)-f(y_{(\bar x, \bar t)}(s), \gamma[b^*](s))\cdot D_x\phi(y_{(\bar x, \bar t)}(s), m(\cdot, s; b^*, \bar t, \bar m), s)\\
+\langle D_m\phi(y_{(\bar x, \bar t)}(s), m(\cdot, s; b^*, \bar t, \bar m), s), \div (b^*m(\cdot, s; b^*, \bar t, \bar m))\rangle_{L^2(\Omega_1(T))}\\
-\ell(y_{(\bar x, \bar t)}(s), m(\cdot, s; b^*, \bar t, \bar m), s, \gamma[b^*](s), b^*)\leq-\frac{\theta}{2}
\end{multline*}
for every $\bar t\leq s<\tau$, where $y_{(\bar x, \bar t)}(s)=y_{(\bar x, \bar t)}(s; \gamma[b^*])$. By integrating from $\bar t$ to $\tau$, and by \eqref{ode1} and \eqref{contequationformass}, we get
\begin{multline*}
\phi(\bar x, \bar m, \bar t)-\phi(y_{(\bar x, \bar t)}(\tau), m(\cdot, \tau; b^*, \bar t, \bar m), \tau)-\int_{\bar t}^{\tau}\ell(y_{(\bar x, \bar t)}(s), m(\cdot, s; b^*, \bar t, \bar m), s, \gamma[b^*](s), b^*)ds\\
\leq -\frac{\theta(\tau-\bar t)}{4},
\end{multline*}
for $\tau$ sufficiently close to $\bar t$. From
$$
\phi(\bar x, \bar m, \bar t)-\phi(y_{(\bar x, \bar t)}(\tau), m(\cdot, \tau; b^*, \bar t, \bar m), \tau)\geq \underbar V(\bar x, \bar m, \bar t)-\underbar V(y_{(\bar x, \bar t)}(\tau), m(\cdot, \tau; b ^*, \bar t, \bar m), \tau),
$$
we obtain
\begin{multline*}
\underbar V(y_{(\bar x, \bar t)}(\tau), m(\cdot, \tau; b^*, \bar t, \bar m), \tau)+\int_{\bar t}^{\tau}\ell(y_{(\bar x, \bar t)}(s), m(\cdot, s; b^*, \bar t, \bar m), s, \gamma[b^*](s), b^*)ds\\
\geq \frac{\theta(\tau-\bar t)}{4}+\underbar V(\bar x, \bar m, \bar t)
\end{multline*}
and thus
\begin{multline*}
\inf_{\gamma\in\Gamma(\bar t)}\sup_{\beta\in\cB(\bar t)}\Bigg\{\int_{\bar t}^{\tau}\ell(y_{(\bar x, \bar t)}(s), m(\cdot, s; \beta, \bar t, \bar m), s, \gamma[\beta](s), \beta(\cdot, s))ds\\
+\underbar V(y_{(\bar x, \bar t)}(\tau; \gamma[\beta]), m(\cdot, \tau; \beta, \bar t, \bar m), \tau)\Bigg\}>\underbar V(\bar x, \bar m, \bar t),
\end{multline*}
which contradicts \eqref{dpp}. This completes the proof that $\underbar V$ is a viscosity solution (see Definition \ref{viscositysolutiondef}) of \eqref{HJI}. 
\end{proof}
Now we want to characterize the lower value function of the problem as the unique solution of the HJI equation \eqref{HJI}.
\begin{theorem}
\label{teoconfronto}
Assume the hypotheses of Theorem \ref{esistenza}. Let $u_1, u_2$ be bounded and uniformly continuous functions and, respectively, viscosity sub- and supersolution (see Definition \ref{viscositysolutiondef}) of
$$
-V_t+H(x, m, t, D_xV, D_mV)=0\quad\text{in }\tilde X.
$$
Therefore if $u_1\leq u_2$ on $\Omega_2(T)\times X(T)\times\{T\}$, then $u_1\leq u_2$ in $\tilde X$.
\end{theorem}
For the proof, we need the following lemmas.
\begin{lemma}
\label{lemmaham}
For all $\xi>0$, $p\in\R^d$, $(x_1, m_1, t_1), (x_2, m_2, t_2)\in\tilde X$ and $q=\frac{2(m_1-m_2)}{\xi^2}$, we have
\begin{multline*}
\left|H\left(x_1, m_1, t_1, p, \frac{2(m_1-m_2)}{\xi^2}\right)-H\left(x_2, m_2, t_2, p, \frac{2(m_1-m_2)}{\xi^2}\right)\right|\\
\leq L\|x_1-x_2\|\|p\|+M\frac{\|m_1-m_2\|^2_{L^2(\Omega_1(T))}}{\xi^2}\\
+\omega_{\ell}\left(\|x_1-x_2\|+\|m_1-m_2\|_{H^1(\Omega_1(T))}+|t_1-t_2|\right),
\end{multline*}
where we recall that $L$ is the Lipschitz constant of $f(\cdot, a)$ and $\omega_{\ell}$ is the modulus of continuity of the running cost $\ell$ (see \S\ref{dynprogreg}).
\end{lemma}
\begin{proof}
Let us fix $b'\in\tilde\cB$ such that
\begin{multline*}
H\left(x_1, m_1, t_1, p, \frac{2(m_1-m_2)}{\xi^2}\right)\\
\geq -f(x_1, a)\cdot p+\left\langle \frac{2(m_1-m_2)}{\xi^2}, \div(b'm_1)\right\rangle_{L^2(\Omega_1(T))}-\ell(x_1, m_1, t_1, a, b')\quad\text{for every }a\in A.
\end{multline*}
Then take $a'\in A$ such that
\begin{multline*}
H\left(x_2, m_2, t_2, p, \frac{2(m_1-m_2)}{\xi^2}\right)\\
\leq-f(x_2, a')\cdot p+\left\langle \frac{2(m_1-m_2)}{\xi^2}, \div(b'm_2)\right\rangle_{L^2(\Omega_1(T))}-\ell(x_2, m_2, t_2, a', b').
\end{multline*}
Therefore
\begin{multline*}
\left|H\left(x_2, m_2, t_2, p, \frac{2(m_1-m_2)}{\xi^2}\right)-H\left(x_1, m_1, t_1, p, \frac{2(m_1-m_2)}{\xi^2}\right)\right|\\
\leq\left|-f(x_2, a')\cdot p+\left\langle \frac{2(m_1-m_2)}{\xi^2}, \div(b'm_2)\right\rangle_{L^2(\Omega_1(T))}-\ell(x_2, m_2, t_2, a', b')\right.\\
\left.+f(x_1, a')\cdot p-\left\langle \frac{2(m_1-m_2)}{\xi^2}, \div(b'm_1)\right\rangle_{L^2(\Omega_1(T))}+\ell(x_1, m_1, t_1, a', b')\right|\\
\leq\Bigg|(f(x_1, a')-f(x_2, a'))\cdot p+\left\langle \frac{2(m_1-m_2)}{\xi^2}, \div(b'(m_2-m_1))\right\rangle_{L^2(\Omega_1(T))}\\
+\omega_{\ell}\left(\|x_1-x_2\|+\|m_1-m_2\|_{H^1(\Omega_1(T))}+|t_1-t_2|\right)\Bigg|\\
\leq \|f(x_1, a')-f(x_2, a')\|\|p\|+\left|\frac{1}{\xi^2}\int_{\Omega_1(T)}\div(b')(m_1-m_2)^2dx\right|\\
+\omega_{\ell}\left(\|x_1-x_2\|+\|m_1-m_2\|_{H^1(\Omega_1(T))}+|t_1-t_2|\right)\\
\leq L\|x_1-x_2\|\|p\|+M\frac{\|m_1-m_2\|^2_{L^2(\Omega_1(T))}}{\xi^2}\\
+\omega_{\ell}\left(\|x_1-x_2\|+\|m_1-m_2\|_{H^1(\Omega_1(T))}+|t_1-t_2|\right),\\
\end{multline*}
where in the second-to-last inequality we have integrated by parts and used the fact that $m_1,m_2$ vanish at $\partial\Omega_1(T)$. 
\end{proof}
\begin{lemma}
\label{lemmaext}
If $u_1, u_2$ are respectively sub and supersolution (see Definition \ref{viscositysolutiondef}) of \eqref{HJI} in $\tilde X$, they are also sub and supersolution in $\Omega_2(t)\times X(t)\times[0, T[$. More precisely, the viscosity inequalities hold if the maximum or minimum points are on $\Omega_2(0)\times X(0)\times\{0\}$.
\end{lemma}
\begin{proof}
The proof goes similarly as the one of \cite{BCD}, Ch. II, \S2, Lemma 2.10.
\end{proof}
\par\smallskip\noindent
{\it Proof (of Theorem \ref{teoconfronto}).}
The aim of the proof is to show that $G=\sup_{\tilde X}(u_1-u_2)$ is less or equal to $0$. We argue by contradiction assuming that $G>0$.
\par
In order to simplify the proof, we change $u_1$ with $(u_1(x, m, t))_{\eta}:=u_1(x, m, t)-\eta t$ for some $\eta>0$ sufficiently small, in such a way that, without loss of generality, we may assume that $u_1$ is a strict subsolution (see Definition \ref{viscositysolutiondef}) of \eqref{HJI} since $(u_1)_{\eta}$ is a subsolution of 
\begin{equation}
\label{eqreduced}
-\frac{\partial}{\partial t}(u_1)_{\eta}+H(x, t, m, D_x(u_1)_{\eta}, D_m(u_1)_{\eta})\leq-\eta<0\quad\text{in}\ \tilde X.
\end{equation}
To perform the proof, it is sufficient to show that $(u_1)_{\eta}\leq u_2$ in $\tilde X$ for any $\eta$ and then to let $\eta$ tend to $0$. Notice also that we still have $(u_1)_{\eta}\leq u_2$ on $\Omega_2(T)\times X(T)\times\{T\}$. To simplify the notation, and since the proof is clearly reduced to compare $(u_1)_{\eta}$ and $u_2$, we drop the $\eta$ and use the notation $u_1$ instead of $(u_1)_{\eta}$.
\par
Next, we consider the difficulty with $\Omega_2(0)\times X(0)\times\{0\}$: a priori, we do not know if $u_1\leq u_2$ on this part of the boundary and a maximum point of $u_1-u_2$ can be located there. It is solved by Lemma \ref{lemmaext}.
\par
Now, since $u_1$ and $u_2$ are not smooth, we need an argument in order to be able to use the definition of viscosity solution (see Definition \ref{viscositysolutiondef}). We introduce then the test function
\begin{multline*}
\Psi_{\ep, \xi, \a}(x', m', t', x'', m'', t'')=u_1(x', m', t')-u_2(x'', m'', t'')\\
-\frac{\|x'-x''\|^2}{\ep^2}-\frac{\|m'-m''\|^2_{L^2(\Omega_1(T))}}{\xi^2}-\frac{|t'-t''|^2}{\a^2}.
\end{multline*}
Since $\Psi_{\ep, \xi, \a}$ is continuous on $\tilde X\times\tilde X$, there exists a maximum point $(\tilde x', \tilde m', \tilde t', \tilde x'', \tilde m'', \tilde t'')$, and we set $\bar G:=\Psi_{\ep, \xi, \a}(\tilde x', \tilde m', \tilde t', \tilde x'', \tilde m'', \tilde t'')$. We prove that
\begin{itemize}
\item[$(1)$] when $\ep, \xi, \s\to0$, then $\bar G\to G$;
\item[$(2)$] $u_1(\tilde x', \tilde m', \tilde t')-u_2(\tilde x'', \tilde m'', \tilde t'')\to G$ as $\ep, \xi, \a\to0$;
\item[$(3)$] we have
$$
\frac{\|x'-x''\|^2}{\ep^2},\frac{\|m'-m''\|^2_{L^2(\Omega_1(T))}}{\xi^2},\frac{|t'-t''|^2}{\a^2}\to0\quad\text{as }\ \ep, \xi, \a\to0;
$$
\item[$(4)$] $(\tilde x', \tilde m', \tilde t'), (\tilde x'', \tilde m'', \tilde t'')\in\tilde X$.
\end{itemize}
Since $(\tilde x', \tilde m', \tilde t', \tilde x'', \tilde m'', \tilde t'')$ is a maximum point of $\Psi_{\ep, \xi, \a}$, for any $(x', m', t'), (x'', m'', t'')\in\tilde X$, we have
\begin{multline*}
u_1(x', m', t')-u_2(x'', m'', t'')-\frac{\|x'-x''\|^2}{\ep^2}-\frac{\|m'-m''\|^2_{L^2(\Omega_1(T))}}{\xi^2}-\frac{|t'-t''|^2}{\a^2}\\
\leq u_1(\tilde x', \tilde m', \tilde t')-u_2(\tilde x'', \tilde m'', \tilde t'')-\frac{\|\tilde x'-\tilde x''\|^2}{\ep^2}-\frac{\|\tilde m'-\tilde m''\|^2_{L^2(\Omega_1(T))}}{\xi^2}-\frac{|\tilde t'-\tilde t''|^2}{\a^2}=\bar G.
\end{multline*}
We choose $x'=x''$, $m'=m''$ and $t'=t''$ in the left-hand side to get
$$
u_1(x', m', t')-u_2(x', m', t')\leq\bar G\quad\text{for all }(x', m', t')\in\tilde X,
$$
and, by taking the supremum over $(x', m', t')$, we obtain the inequality $G\leq\bar G$.
\par
Since $u_1$ and $u_2$ are bounded, we can set $R:=\max\{\|u_1\|_{\infty}, \|u_2\|_{\infty}\}$ and we also have, by arguing in an analogous way,
\begin{multline*}
G\leq u_1(\tilde x', \tilde m', \tilde t')-u_2(\tilde x'', \tilde m'', \tilde t'')-\frac{\|\tilde x'-\tilde x''\|^2}{\ep^2}-\frac{\|\tilde m'-\tilde m''\|^2_{L^2(\Omega_1(T))}}{\xi^2}-\frac{|\tilde t'-\tilde t''|^2}{\a^2}\\
\leq 2R-\frac{\|\tilde x'-\tilde x''\|^2}{\ep^2}-\frac{\|\tilde m'-\tilde m''\|^2_{L^2(\Omega_1(T))}}{\xi^2}-\frac{|\tilde t'-\tilde t''|^2}{\a^2}.
\end{multline*}
Recalling that $G>0$, we deduce 
\begin{equation}
\label{minoreR}
\frac{\|\tilde x'-\tilde x''\|^2}{\ep^2}+\frac{\|\tilde m'-\tilde m''\|^2_{L^2(\Omega_1(T))}}{\xi^2}+\frac{|\tilde t'-\tilde t''|^2}{\a^2}\leq 2R.
\end{equation}
In particular, $\|\tilde x'-\tilde x''\|, \|\tilde m'-\tilde m''\|_{L^2(\Omega_1(T))}, |\tilde t'-\tilde t''|\to0$ as $\ep, \xi, \a\to0$.
\par
Now, since $\tilde X$ is compact in $\R^d\times L^2(\R^d)\times[0, T]$, we may assume, without loss of generality, that the points $(\tilde x', \tilde m', \tilde t'), (\tilde x'', \tilde m'', \tilde t'')$ converge to the same one: we deduce from this property that
$$
\liminf (u_1(\tilde x', \tilde m', \tilde t')-u_2(\tilde x'', \tilde m'', \tilde t''))\leq\limsup(u_1(\tilde x', \tilde m', \tilde t')-u_2(\tilde x'', \tilde m'', \tilde t''))\leq G,
$$
and the inequality
\begin{multline}
\label{disugmax}
G\leq u_1(\tilde x', \tilde m', \tilde t')-u_2(\tilde x'', \tilde m'', \tilde t'')-\frac{\|\tilde x'-\tilde x''\|^2}{\ep^2}-\frac{\|\tilde m'-\tilde m''\|^2_{L^2(\Omega_1(T))}}{\xi^2}-\frac{|\tilde t'-\tilde t''|^2}{\a^2}\\
\leq u_1(\tilde x', \tilde m', \tilde t')-u_2(\tilde x'', \tilde m'', \tilde t'')
\end{multline}
implies that necessarily $\lim (u_1(\tilde x', \tilde m', \tilde t')-u_2(\tilde x'', \tilde m'', \tilde t''))=G$. Finally, because again of \eqref{disugmax},
$$
\bar G=u_1(\tilde x', \tilde m', \tilde t')-u_2(\tilde x'', \tilde m'', \tilde t'')-\frac{\|\tilde x'-\tilde x''\|^2}{\ep^2}-\frac{\|\tilde m'-\tilde m''\|^2_{L^2(\Omega_1(T))}}{\xi^2}-\frac{|\tilde t'-\tilde t''|^2}{\a^2}\to G
$$
and, since $u_1(\tilde x', \tilde m', \tilde t')-u_2(\tilde x'', \tilde m'', \tilde t'')\to G$, we immediately get
$$
-\frac{\|\tilde x'-\tilde x''\|^2}{\ep^2}-\frac{\|\tilde m'-\tilde m''\|^2_{L^2(\Omega_1(T))}}{\xi^2}-\frac{|\tilde t'-\tilde t''|^2}{\a^2}\to0,
$$
and we have proved $(1)$, $(2)$ and $(3)$.
\par
For $(4)$, it is enough to observe that, if the point $(x', m', t')$ is a limit of a subsequence of $(\tilde x', \tilde m', \tilde t'),(\tilde x'', \tilde m'', \tilde t'')$, then $u_1(x', m', t')-u_2(x', m', t')=M>0$ and therefore $(x', m', t')$ cannot be on $\Omega_2(T)\times X(T)\times\{T\}$.
\par
Now we assume that $\ep, \xi, \a$ are sufficiently small in such a way that $(4)$ holds. Since $(\tilde x', \tilde m', \tilde t', \tilde x'', \tilde m'', \tilde t'')$ is a maximum point of $\Psi_{\ep, \xi, \a}$, $(\tilde x', \tilde m', \tilde t')$ is a maximum point of the function
$$
(x', m', t')\longmapsto u_1(x', m', t')-\phi^1(x', m', t'),
$$
where
$$
\phi^1(x', m', t')=u_2(\tilde x'', \tilde m'', \tilde t'')+\frac{\|x'-\tilde x''\|^2}{\ep^2}+\frac{\|m'-\tilde m''\|^2_{L^2(\Omega_1(T))}}{\xi^2}+\frac{|t'-\tilde t''|^2}{\a^2},
$$
but $u_1$ is a viscosity subsolution (see Definition \ref{viscositysolutiondef}) of \eqref{eqreduced} and $(\tilde x', \tilde m', \tilde t')\in\tilde X$. Therefore
\begin{multline*}
-\phi^1_t(\tilde x', \tilde m', \tilde t')+H(\tilde x', \tilde m', \tilde t', D_x\phi^1(\tilde x', \tilde m', \tilde t'), D_m\phi^1(\tilde x', \tilde m', \tilde t'))\\
=\frac{2(\tilde t''-\tilde t')}{\a^2}+H\left(\tilde x', \tilde m', \tilde t', \frac{2(\tilde x'-\tilde x'')}{\ep^2}, \frac{2(\tilde m'-\tilde m'')}{\xi^2}\right)\leq-\eta.
\end{multline*}
In the same way, $(\tilde x'', \tilde m'', \tilde t'')$ is a maximum point of the function
$$
(x'', m'', t'')\longmapsto-u_2(x'', m'', t'')+\phi^2(x'', m'', t''),
$$
where
$$
\phi^2(x'', m'', t'')=u_1(\tilde x', \tilde m', \tilde t')-\frac{\|\tilde x'-x''\|^2}{\ep^2}-\frac{\|\tilde m'-m''\|^2_{L^2(\Omega_1(T))}}{\xi^2}-\frac{|\tilde t'-t''|^2}{\a^2},
$$
and hence $(\tilde x'', \tilde m'', \tilde t'')$ is a minimum point of the function $u_2-\phi^2$. But $u_2$ is a viscosity supersolution (see Definition \ref{viscositysolutiondef}) of \eqref{HJI} and $(\tilde x'', \tilde m'', \tilde t'')\in\tilde X$, and therefore
\begin{multline*}
-\phi^2_t(\tilde x'', \tilde m'', \tilde t'')+H(\tilde x'', \tilde m'', \tilde t'', D_x\phi^2(\tilde x'', \tilde m'', \tilde t''), D_m\phi^2(\tilde x'', \tilde m'', \tilde t''))\\
=\frac{2(\tilde t''-\tilde t')}{\a^2}+H\left(\tilde x'', \tilde m'', \tilde t', \frac{2(\tilde x'-\tilde x'')}{\ep^2}, \frac{2(\tilde m'-\tilde m'')}{\xi^2}\right)\geq0.
\end{multline*}
Then we subtract the two viscosity inequalities to obtain
\begin{multline*}
H\left(\tilde x', \tilde m', \tilde t', \frac{2(\tilde x'-\tilde x'')}{\ep^2}, \frac{2(\tilde m'-\tilde m'')}{\xi^2}\right)\\
-H\left(\tilde x'', \tilde m'', \tilde t'', \frac{2(\tilde x'-\tilde x'')}{\ep^2}, \frac{2(\tilde m'-\tilde m'')}{\xi^2}\right)\leq-\eta.
\end{multline*}
Now, by using Lemma \ref{lemmaham}, we get
\begin{multline*}
H\left(\tilde x', \tilde m', \tilde t', \frac{2(\tilde x'-\tilde x'')}{\ep^2}, \frac{2(\tilde m'-\tilde m'')}{\xi^2}\right)-H\left(\tilde x'', \tilde m'', \tilde t'', \frac{2(\tilde x'-\tilde x'')}{\ep^2}, \frac{2(\tilde m'-\tilde m'')}{\xi^2}\right)\\
\leq 2L\frac{\|\tilde x'-\tilde x''\|^2}{\ep^2}+M\frac{\|\tilde m'-\tilde m''\|^2_{L^2(\Omega_1(T))}}{\xi^2}\\
+\omega_{\ell}\left(\|\tilde x'-\tilde x''\|+\|\tilde m'-\tilde m''\|_{H^1(\Omega_1(T))}+|\tilde t'-\tilde t''|\right)\leq-\eta.
\end{multline*}
But, on one hand, $\|\tilde x'-\tilde x''\|, \|\tilde m'-\tilde m''\|_{L^2(\Omega_1(T))}, |\tilde t'-\tilde t''|\to0$ as $\ep, \xi, \a\to0$ (see \eqref{minoreR}) and hence also $\|\tilde m'-\tilde m''\|_{H^1(\Omega_1(T))}\to0$ as $\xi\to0$ since \eqref{modcontm} holds. On the other hand
$$
2L\frac{\|\tilde x'-\tilde x''\|^2}{\ep^2}+M\frac{\|\tilde m'-\tilde m''\|^2_{L^2(\Omega_1(T))}}{\xi^2}\to0\quad\text{when }\ep,\xi\to0.
$$
The above inequality leads us to a contradiction.
$\hfill\square$
\section{Comments on two one-dimensional examples}
\label{esempio1dim}
In this section, we exhibit two one-dimensional examples ($d=1$) and we somehow provide some comments on them, trying to have information on the problems from the corresponding Isaacs equations. However, we have to say that, in the first example, our comments will be not exhaustive and are to be considering as guess for the understanding of the optimal behavior of the players. Indeed, even in simple situations, the differential game between a player and a mass, as studied in the current paper, presents many non-trivial aspects and so many possible admissible and optimal behaviors which are much sensible to the initial conditions. In particular, this happens when the optimal behavior for the mass involves a change of the shape of the function $m$, that is it is not a rigid movement. The first example below concerns the case where the optimal behavior involves such a change. The second one concerns a case where such a change is not involved, the optimal behavior of the mass is a rigid movement and then it is easier to guess the Frechet differential w.r.t. $m$ of the value.
\subsection{First example}
We consider the following one-dimensional example. The dynamics of the single player is given by $x'=\alpha$ with $\alpha:[0,T]\lra[-c,c]=A$ the measurable control and $c>0$. The control $\beta$ for the mass is given by $\beta:[0,T]\lra{\cal B}$ with 
$$
{\cal B}=\left\{b:\mathbb{R}\lra\mathbb{R}:\|b\|_\infty\le c, \|b_x\|_\infty\le c_1\right\},
$$
where $c$ is as for the single player and $c_1$ is such that $b$ can assume values from $c$ to $-c$, or vice-versa, in an interval of length $2r$. Here $r>0$ is a datum of the problem, being the final cost defined as 
$$
\Psi(x,m)=\int_{x-r}^{x+r}m(\xi)d\xi.
$$
The differential game is indeed of the Mayer type: we only have a final cost and no running cost, $\ell\equiv0$. The lower value function is then defined as
\begin{equation}
\label{eq:lower_value_example}
\underbar V(x,m,t)=\inf_\gamma\sup_\beta\int_{x(T)-r}^{x(T)+r}m(\xi,T)d\xi,
\end{equation}
and the Hamilton-Jacobi-Isaacs problem is
\begin{equation}
\label{eq:H_example}
\begin{cases}
-\underbar V_t+\sup_{a\in[-c,c]}\left\{-a\underbar V_x\right\}+\inf_{b\in{\cal B}}\left\{\langle D_m\underbar V,(bm)_\xi\rangle_{L^2}\right\}=0,\\
\underbar V(x,m,T)=\Psi(x,m)
\end{cases}.
\end{equation}
Note that the Isaacs equation is decoupled in the controls (the Isaacs condition holds: the lower Isaacs equation and the upper Isaacs equation coincide) and hence it is reasonable to expect that the lower value $\underbar V$ will turn out to be equal to the upper value function and so to the min-max equilibrium of a possible differential game, where also the mass uses non-anticipating strategies (see the comments in the Introduction). In the sequel we then denote the lower value function \eqref{eq:lower_value_example} simply by $\text{V}$.
\par
We now make the following assumption (whose validity and reasonableness are discussed below): the value function is given by
\begin{equation}
\label{eq:guess}
\text{V}(x,m,t)=\int_{x-r-h_\ell(t)}^{x+r+h_r(t)}m(\xi)d\xi,
\end{equation}
where the functions $h_\ell,h_r$ ($\ell$ stands for ``left'' and $r$ for ``right'') are two time-dependent absolutely continuous functions satisfying the backward ordinary differential equations
\begin{equation}
\label{eq:h}
\begin{cases}
h'_r(t)=-2c{\cal H}\Big(m(x-r-h_\ell(t))-m(x+r+h_r(t))\Big),\\
h_r(T)=0,\\
h'_\ell(t)=-2c{\cal H}\Big(m(x+r+h_r(t))-m(x-r-h_\ell(t))\Big),\\
h_\ell(T)=0,\\
h'_\ell(t)h'_r(t)=0,\\
h'_\ell(t)+h'_r(t)=-2c
\end{cases},
\end{equation}
where the Heaviside function $\cal H$ is defined as ${\cal H}(\xi)=1$, when $\xi>0$, and ${\cal H}(\xi)=0$, when $\xi<0$. In the case $\xi=0$ the Heaviside function is non-defined or possibly multi-valued, but the last two lines of \eqref{eq:h} imply that, when $m(x+r+h_r(t))-m(x-r-h_\ell(t))=0$, the quantities $h'_r(t),h'_\ell(t)$ are equal to zero and equal to $-2c$, respectively.
\par
Let us assume that a solution of \eqref{eq:h} exists and prove that \eqref{eq:guess} gives a solution of \eqref{eq:H_example}. The final condition is easily verified. We have
\begin{equation}
\label{calcoli}
\begin{array}{ll}
\displaystyle
\text{V}_t(x,m,t)=m(x+r+h_r(t))h'_r(t)+m(x-r-h_\ell(t))h'_\ell(t),\\
\displaystyle
\text{V}_x(x,m,t)=m(x+r+h_r(t))-m(x-r-h_\ell(t)),\\
\displaystyle
D_m\text{V}(x,m,t)=\chi_{[x-r-h_\ell(t),x+r+h_r(t)]},
\end{array}
\end{equation}
where $\chi_{[x-r-h_\ell(t),x+r+h_r(t)]}$ is the characteristic function of the interval $[x-r-h_\ell(t),x+r+h_r(t)]$. Plugging \eqref{calcoli} into \eqref{eq:H_example} we get
\begin{multline}
\label{eq:H_V_example}
-m(x+r+h_r(t))h'_r(t)-m(x-r-h_\ell(t))h'_\ell(t)\\
+c\Big|m(x+r-h_r(t))-m(x-r-h_\ell(t))\Big|-c\Big(m(x+r+h_r(t))+m(x-r-h_\ell(t))\Big)=0,
\end{multline}
where in the last line we have used
\begin{multline*}
\inf_{b\in{\cal B}}\langle D_mV,(bm)_\xi\rangle=\inf_{b\in{\cal B}}\int_{x-r-h_\ell(t)}^{x+r+h_r(t)}(bm)_\xi d\xi\\
=\inf_{b\in{\cal B}}\Big(b(x+r+h_r(t))m(x+r+h_r(t))-b(x-r-h_\ell(t))m(x-r-h_\ell(t))\Big)\\
=-c\Big(m(x+r+h_r(t))+m(x-r-h_\ell(t))\Big)
\end{multline*}
because $m\ge0$, $-c\le b\le c$ and $b$ can assume values from $c$ to $-c$, or vice-versa, in any interval of length larger than $2r$ (as the interval $[x-r-h_\ell(t),x+r+h_r(t)]$ certainly is). Since $(h_\ell,h_r)$ satisfies \eqref{eq:h}, if $m(x+r-h_r(t))-m(x-r-h_\ell(t))>0$, then \eqref{eq:H_V_example} becomes
\begin{multline*}
2cm(x-r-h_\ell(t))+c\Big(m(x+r-h_r(t))-m(x-r-h_\ell(t))\Big)\\
-c\Big(m(x+r+h_r(t))+m(x-r-h_\ell(t))\Big)=0.
\end{multline*}
Similarly if $m(x+r-h_r(t))-m(x-r-h_\ell(t))<0$. If instead $m(x+r-h_r(t))-m(x-r-h_\ell(t))=0$, for any possible choice of $(h'_\ell(t),h'_r(t))=(-2c,0)$ or $(h'_\ell(t),h'_r(t))=(0,-2c)$, we get \eqref{eq:H_V_example} vanishing.
\par\smallskip
{\it Modeling considerations}. We look from the point of view of the mass, that is looking for its optimal strategy. Let $t\longmapsto x(t)$ be the trajectory chosen by the single player starting from $x_0$, say at $t=0$. Note that the goal of the single player is to minimize the quantity of mass' agents inside the interval $[x(T)-r,x(T)+r]$ at the final time $T$, whereas the goal of the mass is to maximize it.
The agents of the mass that are already inside the interval $[x_0-r,x_0+r]$ at $t=0$ will move jointly with the single player. The other agents will try to enter the interval $[x(t)-r,x(t)+r]$ at a certain time $t\le T$ and then to stay in $[x(\tau)-r,x(\tau)+r]$ for all remaining $t\le\tau\le T$. The single player cannot avoid the agents entering its reference interval. However, having the same maximal velocity $c$ it can avoid the ones coming from one of the two sides: just moving in one of the two directions at the maximal velocity. Note that, in this case, the agents coming from the other side are entering with a relative velocity of value $2c$, because they move at the maximal velocity $c$ towards $x(t)$. We can adopt a ``static point of observation'' where the player is not moving and the initial mass is not changing, but instead the interval of reference for the player is growing with left and right extremals moving as $t\longmapsto x_0-r-h_\ell(t)$ and $t\longmapsto x_0+r+h_r(t)$, for two suitably chosen time dependent function $h_\ell,h_r$. If such functions, from a backward point of view, satisfy \eqref{eq:h}, then they are optimal and \eqref{eq:guess} gives the value function. Moreover, an optimal strategy for the mass, still in this static point of view, is, at time $t$, 
\begin{equation}
\label{eq:mass_strategy}
b(\xi)=
\begin{cases}
c&\mbox{if } \xi\le x-r-h_\ell(t),\\
linear&\mbox{if } x-r-h_\ell(t)\le\xi\le x+r+h_r(t),\\
-c&\mbox{if } \xi\ge x+r+h_r(t),
\end{cases}
\end{equation}
as indeed detected by the last line of \eqref{eq:H_V_example} and the comments after it.  Note that the ``static point of view'', as well as being a classical approach for pursuit-evasion games (\cite{isa, pattur}, adopting the reference frame of one of the two players), seems more convenient due to the feature of the game as ``accumulation of agents around the player''. However, from a non-static point of view, \eqref{eq:mass_strategy} just means, for the single mass' agent being in the position $\xi\in\mathbb{R}$ at time $t$: if you are at the right of the interval $[x(t)-r,x(t)+r]$, then move with maximal velocity $c$ towards the left; if you are at the left of the interval $[x(t)-r,x(t)+r]$, then move with maximal velocity $c$ towards the right; if you are already inside the interval $[x(t)-r,x(t)+r]$, then move in order to stay inside the intervals $[x(\tau)-r,x(\tau)+r]$ for all subsequent $\tau$.
\par\smallskip
{\it Comments on the hypothesis \eqref{eq:guess} (and \eqref{eq:h}, existence of the solution)}. The hypothesis \eqref{eq:guess}, despite it enlightens some good points of the model, presents some critical issues. The hypothesis that the functions $h_\ell,h_r$ only depend on $t$ is not generally satisfied. Indeed those functions satisfying \eqref{eq:h} are somehow an optimal feedback for the single player and then they should depend on $x$ and $m$ too. In the equations \eqref{eq:h} there should also be the derivative of $h_\ell,h_r$ with respect to $x$ and their Fr\'echet differential with respect to $m$. This also reflects on the fact that a solution $(h_\ell,h_r)$ of \eqref{eq:h}, as it is written above, may not exists, for some profiles of $m$. The model studied in this section and the proposed solution are then just an ansatz of what the general situation should look like. Then, this simple model, which is representative of the overall problem and its difficulties, already deserves further future studies. In any case, whenever \eqref{eq:h} has a solution for any $x$ and $m$, then \eqref{eq:guess} solves \eqref{eq:H_example} and, by uniqueness, is the value function.
\par\smallskip
{\it A simple possible situation with $h_\ell,h_r$ depending only on $t$.} We end this subsection describing a simple favorable situation. Before that we stress again that the goal of the present subsection is just to have some guess on the optimal behavior which, even in this easy example, presents too many non-trivial aspects to be evaluated with mere ``hand-based'' argument. If the initial distribution $m$ is non-increasing in $[x-r-2cT,x+r+2cT]$ for some $x\in\mathbb{R}$, then, for that $x$, \eqref{eq:h} has a solution given by $h_\ell\equiv0$ and $h_r(t)=2c(T-t)$. The value function is then
$$
\text{V}(x,m,t)=\int_{x-r}^{x+r+2c(T-t)}m(\xi)d\xi,
$$
which indeed means that, in a non-static point of view, the player, starting from $x$, moves at the maximal velocity $c$ towards the right and the mass adopts the strategy described above. Moreover, if we suppose that $m$ is strictly decreasing in $[x-r-2cT-\sigma,x+r+2cT+\sigma]$ for some $\sigma>0$, then, even if we perturb $x$ and $m$ a bit (remember that we are working in the set $X$ as in Section \ref{HJIforV} where the convergence of $m$ is also a uniform convergence), the optimal behavior $(h_\ell,h_r)$ does not change, i.e. it is not affected by such a perturbation. This means that their derivatives with respect to $x$ and $m$ are null and that they are still solution of \eqref{eq:h}.
\subsection{Second example}
Let us consider a one-dimensional Mayer problem, where the dynamics of the single player, its measurable controls and the set $\cB$ of the controls of the mass are the same as the example in the subsection \S5.1, and the only change is in the final cost which is now given by
$$
\psi(x, m)=\left(x-\int_{\R}\xi m(\xi)d\xi\right)^2.
$$
Besides the hypotheses on \S4, we also consider only initial states for the mass such that $\int_{\R}md\xi=1$ and $m\geq0$. Note that such properties are maintained by the solutions of the continuity equation (1) and are compatible with the arguments in \S4. The single player $x$ wants to minimize the cost
$$
J(x, m, t, \a, \b)=\psi(x(T), m(T))
$$
and the mass $m$ wants to maximize it, that is the single player at the final time $T$ wants to be close to the expectation of the distribution of the agents and the contrary for the mass.
\par\medskip\noindent
\par
Again, the problem solved by the value of the game $V$ is
\begin{equation}
\label{hji}
\begin{cases}
-V_t+\max_{a\in[-c, c]}\{-V_xa\}+\inf_{b\in\tilde\cB}\langle D_mV, (bm)_{\xi}\rangle=0,\\
V(x, m, T)=\psi(x, m).
\end{cases}
\end{equation}
We assert that the value is
\begin{equation}
\label{valuefunctionsecondexample}
V(x, m, t)=\psi(x, m)=\left(x-\int_{\R}\xi m(\xi)d\xi\right)^2
\end{equation}
independently of $t$, that is anyone of the two opponents cannot improve its outcome w.r.t. initial one.
\par
We prove that $V$ solves \eqref{hji}. We have
\begin{align*}
V_t&=0\quad\text{for every}\ t;\\
V_x&=2\left(x-\int_{\R}\xi m(\xi)d\xi\right);\\
D_mV&=-2\left(x-\int_{\R}\xi m(\xi)d\xi\right)\xi.
\end{align*}
We set $z:=2\left(x-\int_{\R}\xi m(\xi)d\xi\right)\in\R$. Let us assume that $z\geq0$. Then we have
\begin{multline*}
-V_t+\max_{a\in[-c, c]}\{-V_xa\}+\inf_{b}\langle D_mV, (bm)_{\xi}\rangle=\max_{a\in[-c, c]}\{-za\}+\inf_b\int_{\R}(-z\xi)(bm)_{\xi}d\xi\\
=zc+z\inf_b\left(-\xi bm\Big{|}^{\xi=+\infty}_{\xi=-\infty}+\int_{\R}bmd\xi\right)=zc+z\inf_b\int_{\R}bmd\xi=zc+z\int_{\R}(-c)md\xi=zc-zc=0.
\end{multline*}
Clearly, the final datum is satisfied. The calculation with $z\leq0$ goes similarly. Hence $V$ is the solution of \eqref{hji} and, by uniqueness, it is the value of the game.
\par\smallskip
{\it On the optimal strategies.} The fact that the value of the game is \eqref{valuefunctionsecondexample} and the calculation above show that the optimal strategies, for both opponents, are realized when they move simultaneously at the maximum velocity towards a suitable (left or right) direction. If this does not happen, one of the two opponents can surely improve its outcome at the cost of the other one.


\end{document}